\RequirePackage{fix-cm}
\documentclass[smallextended]{svjour3}       % onecolumn (second format)
\smartqed  % flush right qed marks, e.g. at end of proof
\usepackage{graphicx}
\usepackage{pgf,tikz}
\usetikzlibrary{arrows}
\usepackage{array,multirow,makecell}
\setcellgapes{1pt} \makegapedcells
\usepackage{amssymb}
\usepackage{hyperref}
\usepackage{epstopdf}
\hypersetup{
    colorlinks=true,
    linkcolor=red,
    filecolor=magenta,
    urlcolor=cyan,
}
\usepackage{graphicx}
\usepackage{here}
\usepackage{epsfig}
\usepackage{color}
\usepackage{subfig}
\usepackage[normalem]{ ulem }
\usepackage{soul}
\newcommand{\R}{\mathbb{R}}

\newcommand{\comment}[1]{}
\begin{document}

\title{Spreading speeds and traveling waves for monotone systems of impulsive reaction-diffusion equations: application to tree-grass interactions in fire-prone savannas 
\thanks{The research was supported by the DST/NRF SARChI Chair  in Mathematical Models and Methods in Biosciences and Bioengineering at the University of Pretoria (grant 82770) and National Science Centre, Poland, grant 2017/25/B/ST1/00051.
}
}
% Grants or other notes about the article that should go on the front
% page should be placed within the \thanks{} command in the title
% (and the %-sign in front of \thanks{} should be deleted)
%
% General acknowledgments should be placed at the end of the article.

\subtitle{The paper is dedicated to the memory of Professor H. I. Freedman}

\titlerunning{Impulsive monotone reaction-diffusion systems: traveling waves and spreading speeds}        % if too long for running head

\author{J. Banasiak \and Y. Dumont  \and
        I.V. Yatat Djeumen\footnote{Corresponding author: ivric.yatatdjeumen@up.ac.za}
}

 % if too long for running head

\institute{J. Banasiak, Y. Dumont and I.V. Yatat Djeumen: \at
              University of Pretoria, Department of Mathematics and Applied Mathematics, Pretoria, South Africa.
              %\email{jacek.banasiak@up.ac.za}\\
              %\email{ivric.yatatdjeumen@up.ac.za}%  \\
           \and
          Y. Dumont: \at AMAP, University of Montpellier, CIRAD, CNRS, INRA, IRD, Montpellier, France.\\
              CIRAD, Umr AMAP, Pretoria, South Africa.
            \and J. Banasiak: \at Institute of Mathematics, \L\'{o}d\'{z} University of Technology, \L\'{o}d\'{z}, Poland.  }

\date{Received: date / Accepted: date}

\maketitle

\begin{abstract}
Many systems in life sciences have been modeled by
reaction-diffusion equations. However, under some circumstances,
these biological systems may experience instantaneous and periodic
perturbations (e.g. harvest, birth, release, fire events, etc) such that an
appropriate formalism is necessary, using, for instance, impulsive
reaction-diffusion equations. While several works tackled the issue of traveling waves for monotone reaction-diffusion equations and the computation of spreading speeds, very little has been done in
the case of monotone impulsive reaction-diffusion equations. Based
on vector-valued recursion equations theory, we aim to present in
this paper results that address two main issues of monotone
impulsive reaction-diffusion equations.  First, they  deal with the
existence of traveling waves for monotone systems of impulsive reaction-diffusion equations. Second, they allow the computation of spreading speeds for monotone systems of impulsive reaction-diffusion equations. We apply our methodology to a planar system
of impulsive reaction-diffusion equations that models tree-grass
interactions in fire-prone savannas. Numerical simulations, including
numerical approximations of spreading speeds, are finally provided
in order to illustrate our theoretical results and support the
discussion.


\end{abstract}

\section{Introduction}
In nature, all organisms migrate or disperse to some extent. This
can take diverse forms such as walking, swimming, flying, or
being transported by wind or flowing water, see  Shigesada and Kawasaki \cite{Shigesada1997}. Such a migration, or dispersion, can be
 to some extent related to human activities that bring drastic
changes in the global environment. According to \cite{Shigesada1997}, dispersive movements become
noticeable when an offspring or a seed leaves its natal site,
or when an organism's habitat deteriorates from overcrowding. The spatially explicit ecological theories of such events have  been made possible thanks  to successful development of mathematical models that have played a central role in the description of migrations, see Friedman
\cite{Friedman1964}, Shigesada and Kawasaki    \cite{Shigesada1997}, Okubo and Levin \cite{Okubo2001},
Cantrell and Cosner  \cite{Cantrell2003}, Volpert
\cite{Volpert2014}, Logan   \cite{Logan2008},
\cite{Logan2015}, Perthame   \cite{Perthame2015}, and the references therein.

Mathematical literature dealing with the species' spread mostly relies
on reaction-diffusion equations that assume
that the dispersal is governed by random diffusion and that it, along with the growth processes, take place continuously in time and space,  (Cantrell and
Cosner   \cite{Cantrell2003}, Lewis and Li
\cite{Lewis2012}). This approach  has had a remarkable success in explaining
the rates at which species have invaded large open environments, see Shigesada and Kawasaki   \cite{Shigesada1997}, Okubo and Levin
  \cite{Okubo2001}, Cantrell and Cosner   \cite{Cantrell2003}, Lewis
and Li   \cite{Lewis2012}, Volpert   \cite{Volpert2014},
Logan   \cite{Logan2008},  \cite{Logan2015}, Perthame
  \cite{Perthame2015}. However, it is well-known that
ecological species may experience several phenomena that, depending on circumstances, can be
either time-continuous (growth, death, birth, release, etc.), or
time-discrete (harvest, birth, death, release, etc.),  see also Ma
and Li   \cite{Ma2009}, Dumont and Tchuenche
\cite{Dumont2012}, Yatat et al.   \cite{Yatat2017}, Yatat   \cite{PhDYatatDjeumen2018} and the references therein. In the case of time-discrete perturbations, whose duration is negligible in
 comparison with the duration of the process, it is natural to
assume that these perturbations act instantaneously; that is, in the form of impulses (Lakshmikantham et al.
\cite{Lakshmikantham1989}, Bainov and Simeonov
\cite{Bainov1995}). Hence, there is a need to create a meaningful mathematical framework to analyse models leading to, say, impulsive reaction-diffusion equations.

Before going further, let us make some comments about how
$\tau$-periodic impulsive phenomena are taken into account in
mathematical models. Here, for simplicity of the exposition, we will focus  on
mathematical models depending only on time. We note that  there exist several
possibilities  to model pulse events such as the formalism presented in (Lakshmikantham
et al.  \cite{Lakshmikantham1989}, Bainov and Simeonov
\cite{Bainov1995}, Dumont and Tchuenche   \cite{Dumont2012},
Dufourd and Dumont  \cite{Dufourd2013}, Tchuint\'e Tamen et
al. \cite{Tchuinte2016},   \cite{Tchuinte2017}, Yatat
  \cite{PhDYatatDjeumen2018}, Yatat et al.
\cite{Yatat2017},   \cite{Yatat2018} and the references therein), or the approach used  by Lewis and Li
\cite{Lewis2012}, see also Weinberger
et al.   \cite{Weinberger2002}, Lewis et al.
\cite{Lewis2002}, Li et al.   \cite{Li2005}, Vasilyeva et al.   \cite{Vasilyeva2016}, Fazly et al.   \cite{Fazly2017}, Huang et al.  \cite{Huang2017}, Yatat and Dumont \cite{YatatDumont2018} and the references therein.

Following  \cite{Lakshmikantham1989}, we describe the evolution process by
\begin{equation}\label{lakshmikantham}
\begin{array}{l}
\displaystyle\frac{dx}{dt}=f(t,x),\quad t\neq k\tau, \\
\end{array}
\end{equation}
\begin{equation}\label{Imp-lakshmikantham}
\begin{array}{l}
x(k\tau^+)-x(k\tau^-)=g(k\tau,x(k\tau^-)),\quad k=1,2,3,...,
\end{array}
\end{equation}
where $f,g:\R_+\times\Omega\rightarrow\R^n$, $\Omega\subseteq\R^n$
is an open set and
$$x(k\tau^\pm)=\lim\limits_{\theta\rightarrow0}x(k\tau\pm\theta).$$
Usually, $x$ is assumed to be left-continuous; that is,
$x(k\tau^-)=x(k\tau)$. Let $x(t)=x(t,t_0,x_0)$ be any solution of
(\ref{lakshmikantham}) starting at $(t_0,x_0)$. The evolution
process behaves as follows: the point $p_t=(t,x(t))$ begins its
motion from the initial point $p_{t_0}=(t_0, x_0)$ and moves along
the curve $\{(t,x):t\geq t_0, x=x(t)\}$ until the time
$t_1=\tau>t_0$ at which the point $p_{t_1}=(t_1, x(t_1))$ is
transferred to $p_{t_1^+}=(t_1, x_1^+)$ where
$x_1^+=x(t_1)+g(t_1,x(t_1))$. Then the point $p_t$ continues to move
further along the curve with $x(t)=x(t,t_1,x_1^+)$ as the solution
of (\ref{lakshmikantham}) starting at $p_{t_1}=(t_1,x_1^+)$ until
the next moment $t_2=2\tau>t_1$. Then, once again the point
$p_{t_2}=(t_2, x(t_2))$ is transferred to $p_{t_2^+}=(t_2, x_2^+),$
where $x_2^+=x(t_2)+g(t_2,x(t_2))$. As before, the point $p_t$
continues to move forward with $x(t)=x(t,t_2,x_2^+)$ as the solution
of (\ref{lakshmikantham}) starting at $(t_2,x_2^+)$. Thus, the
evolution process continues forward as long as the solution of
(\ref{lakshmikantham}) exits.

To describe the other approach, we focus on $\tau$-periodic impulsive perturbations. The inter-perturbation
season (i.e. the time between two successive perturbations) has
length $\tau$ (units of time) and at the end of the
inter-perturbation season, a prescribed perturbation occurs. Hence, we can
consider the whole time interval as a succession of
inter-perturbation seasons of length $\tau$. We denote the state
variable at time $t\in [0, \tau]$ during the inter-perturbation
season $k\in\mathbb{N}^*=\{1,2,3,...\}$ by $x_k(t)\in\mathbb{R}^n;$ its
inter-perturbation dynamics is described by
\begin{equation}\label{lewis-formalism}
\begin{array}{l}
\displaystyle\frac{dx_{k}}{dt}=f(t,x_{k}),\quad 0\leq t\leq\tau, \\
\end{array}
\end{equation}
where $f:\R_+\times\Omega\rightarrow\R^n$ is a $\tau$-periodic function and $\Omega\subseteq\R^n$
is an open set. By the end of each season the perturbation is given by an updating condition
\begin{equation}\label{update-lewis-formalism}
\begin{array}{l}
x_{k+1}(0)=G_k[x_k(\tau)],
\end{array}
\end{equation}
where $G_k: \Omega\to \Omega, k=1,2,\ldots,$ are (possibly) nonlinear operators,  with given $x_1(0)\in\Omega.$ Then system
(\ref{lewis-formalism})-(\ref{update-lewis-formalism}) can be written as
\begin{equation}\label{recursion-lewis-formalism}
\begin{array}{l}
x_{k+1}(0)=G_kQ[x_k(0)],\quad k=1,2,3,...
\end{array}
\end{equation}
where $Q$ is the so called time-$\tau$-map operator solution of system
(\ref{lewis-formalism}). We note that there are
analytic links between these two formalisms, see \cite{YatatDumont2018}.

There are several works that considered the impact of impulsive perturbations on the dynamics of a system, both in the space-implicit and space-explicit case (see also Section \ref{section-example}). Very often, impulsive perturbations in space-implicit mathematical models result in the occurrence of periodic solutions in the model (e.g. Ma and Li \cite{Ma2009}, Yatat et al. \cite{Yatat2017} and references therein).  For space-explicit impulsive mathematical models in bounded domains, in addition to periodic solutions that may occur, the issue of minimal domain has been also addressed (e.g. Lewis and Li \cite{Lewis2012}, Yatat and Dumont \cite{YatatDumont2018}). On the other hand, in unbounded domains the problems of the existence of travelling wave solutions and the  computation of the
spreading speeds are hardly addressed (see below).

The study of traveling waves, as well as the computation of
spreading speeds for  monotone systems of reaction-diffusion
equations, have been done by several authors (e.g. Weinberger et al.
 \cite{Weinberger2002}, Lewis et al.   \cite{Lewis2002},
Li et al.   \cite{Li2005}, Volpert   \cite{Volpert2014},
Yatat et al.   \cite{Yatat2017b} and references therein).
However, little is known about the case of impulsive
reaction-diffusion equations. For a scalar impulsive
reaction-diffusion equation, Yatat and Dumont
\cite{YatatDumont2018} considered the Fisher-Kolmogorov-Petrowsky-Piscounov (FKPP) equation and obtained conditions under which an
invasive traveling wave, connecting the extinction equilibrium and
the positive equilibrium may exist (see also Lewis and Li
\cite{Lewis2012}). To the best of our knowledge, the existence of
traveling waves for system of impulsive
reaction-diffusion equation was studied only in Huang et al.
\cite{Huang2017}, and only in a particular case. Precisely, the authors
considered a stage-structured population model, were only one
stage (or state variable) experiences a spatial diffusion, while the
others are stationary. This assumption leads to a partially
degenerate system of impulsive reaction-diffusion equations. Moreover, they
also assumed reaction term for the diffusing state variable  as linear.

The aim of this paper is to give some insights into the existence of traveling wave solutions for monotone systems of
impulsive reaction-diffusion equations as well as the
computation of their spreading speeds. Precisely, we use
the vector-valued recursion theory proposed by Weinberger
et al.   \cite{Weinberger2002} (see also Lewis et al.
\cite{Lewis2002}, Li et al.   \cite{Li2005}, Lewis and Li   \cite{Lewis2012} and the references
therein) to develop a framework that is able to deal
with the existence of
traveling wave solutions for monotone systems of impulsive
reaction-diffusion equations and
the computation of spreading speeds. The  paper is organized as
follows: Section \ref{section-example} deals with a brief review of both space-implicit and space-explicit mathematical models that take into account pulse events.
Section \ref{generic-framework} deals with the presentation of
the framework for the computation of the
spreading speeds and the existence of traveling wave solutions for
monotone systems of impulsive reaction-diffusion equations. Section
\ref{application} deals with the application of this framework to a system of two impulsive reaction-diffusion equations that
models tree-grass interactions in fire-prone savannas.  We also provide some numerical illustrations of our theoretical results and, in particular, we show an  approximation of the spreading speeds.

\section{A brief review of mathematical models describing pulse events
}\label{section-example}
We now provide a brief literature review of both space-implicit and space-explicit impulsive mathematical models.

\subsection{Space-implicit impulsive models}
Several types of perturbations, or instantaneous phenomena, have been considered as pulse events in mathematical models. These include events such as  birth, vaccination, release, harvest, or fire events. The resulting impulsive models were rigorously analyzed by their authors by the well-known theory due to Lakshmikantham et al. \cite{Lakshmikantham1989}, Bainov and Simeonov  \cite{Bainov1995} and \cite{Bainov1989}, or Lakmeche and Arino \cite{Lakmeche2000}. We note that, by using a suitable comparison argument, the standard theory of ordinary differential equations (Hale \cite{Hale1988}, \cite{Hale1980}) can also be used.

\subsubsection{Modelling births as pulse events}

Several authors (Ma and Li \cite{Ma2009},    Wenjun and Jin \cite{Wenjun2007}, Zhang et al. \cite{Zhang2008}) analysed the dynamics of infectious diseases in a population, where births occur periodically as a single pulse and also compared the effects of constant and pulse birth process. Namely, they found that if the birth pulse period is greater than some threshold that depends on the parameters related to the dynamics of the infection, then it is easier for such a population to eliminate the disease than if the birth process is constant. On the other hand, if the birth pulse period is lower than the threshold, then the population with a constant birth process eliminates the disease faster. In other words, when the birth pulse period gets very large, the births become less important and have little effect on the population. Hence, the disease in the population with a pulse birth can be eliminated more easily. However, as the birth pulse period gets very small, the population gives births many times in a very short time period, which have stronger effect on the population than in the case with constant births. Then it becomes more difficult to eradicate the disease in the populations with pulse births (Ma and Li \cite{Ma2009}).

\subsubsection{Modelling vaccinations as pulse events}
Vaccination strategies are designed and applied to anticipate, control, or eradicate infectious diseases. Vaccination strategies include continuous-time vaccination and pulse vaccination. Pulse vaccination strategy (PVS) consists of periodic repetition of impulsive vaccinations in a population, for all the age cohorts. At each vaccination time, a constant fraction of susceptibles is vaccinated.
This kind of vaccination is called impulsive since all the vaccine doses are applied in a time period which is very short with respect to the time scale of the target disease (Ma and Li \cite{Ma2009}). The theoretic analysis of PVS was done by several authors and it was found that this strategy can keep the density of susceptible individuals always below  some threshold above which the epidemics will be recurrent (Agur et al. \cite{Agur1993}). In addition, they showed that PVS may allow for the eradication of the disease with a lower fraction of vaccinated susceptibles, than if the continuous-time vaccination strategy was applied (Agur et al. \cite{Agur1993}, Shulgin et al. \cite{Shulgin1998}, D'Onofrio \cite{Donofrio2002}, Zheng et al. \cite{Zheng2003}, Ma and Li \cite{Ma2009} and references therein).

\subsubsection{Modelling releases as pulse events}
In the framework of biological control of pests or vectors of infectious diseases, the sterile insect technique (SIT) is one of the promising ones. SIT control generally consists in massive releases of sterile insects in the targeted area in order to eliminate, or at least to lower the pest population under a certain threshold (Anguelov et al. \cite{Anguelov2019TIS}). Generally, SIT releases are done periodically. That is why several authors modelled the release process  as a periodic impulsive event, while keeping the continuous-time differential equation framework for the birth, growth, death, or mating process (e.g. White et al. \cite{White2010}, Dumont and Tchuenche \cite{Dumont2012}, Strugarek et al. \cite{Strugarek2019}, Bliman et al. \cite{Bliman2019} and the references therein). Based on the qualitative analysis of their impulsive models, the authors were able to derive meaningful relations between the period and the size of the releases in order to achieve the elimination of the vector or pest population in the long term.

In the context of interacting species such as prey-predator interactions, there exist mathematical models that tackled periodic releases of one of the interacting species (e.g. prey only, predator only). They considered periodic impulses as to model periodic release events. The authors found relations involving the pulse time period and the amount of released species that precluded extinction, in the long term dynamics, of interacting populations (see for instance Zhang et al. \cite{Zhang2005}, Zhao et al. \cite{Zhao2011}).

\subsubsection{Modelling harvests as pulse events}
Other works addressed the question of periodic pulse harvests of interacting populations or, in some cases, of a single population. The authors  aimed to characterize the impact of pulse harvests on the dynamics of the species and they found that, depending on the pulse period and the rate of the harvest, one could avoid the extinction of the population (e.g. Liu et al. \cite{Liu2009}, Zhao et al. \cite{Zhao2011}, Yatat and Dumont \cite{YatatDumont2018}).

\subsubsection{Modelling fires as pulse events in tree-grass interactions in fire-prone savannas}

Maintaining the balance between the grass and the trees in savanna is of utmost importance for both human and animal populations living in such areas. The problem is that in typical circumstances the trees encroach on the grassland making the environment inhabitable for many species. It turns out that periodic fires, either natural or manmade, are one of the way to maintain an acceptable equilibrium.  Thus, several mathematical models have been developed to study tree-grass interactions in fire-prone savannas (see the review of Yatat et al. \cite{Yatat2018}). Some of these models take into account fire as a time-continuous forcing in tree-grass interactions. However, and as pointed out in Yatat et al. \cite{Yatat2018},  it is questionable whether it makes sense to model fire as a
permanent forcing that continuously removes a fraction of the fire sensitive biomass.
Indeed, since several months and even years can pass between two successive fires, they can
be rather considered as instantaneous perturbations of
the savanna ecosystem (see also Yatat \cite{PhDYatatDjeumen2018}, Yatat et al. \cite{Yatat2017}, Tchuint\'e et al. \cite{Tchuinte2017}, \cite{Tchuinte2016}).
Several recent papers have
proposed to model fires either as stochastic events, while keeping the continuous-time differential equation
framework (Baudena et al. \cite{Baudena2010}, Beckkage et al. \cite{Beckage2011}, Synodinos et al. \cite{Synodinos2018}), or by using
a time-discrete model (Higgins et al. \cite{Higgins2008}, Accatino et al. \cite{Accatino2013},  \cite{Accatino2016}, Klimasara and Tyran-Kami\'{n}ska \cite{KT}). However, a drawback of many of
the aforementioned recent time-discrete stochastic
models (Higgins et al. \cite{Higgins2008}, Baudena et al. \cite{Baudena2010}, Beckage et al. \cite{Beckage2011}) is that they
hardly lend themselves to analytical treatment. Thus, on the basis of recent publications (see for instance Yatat et al. \cite{Yatat2018} and the references therein), we consider fires as impulsive time-periodic events. While certainly an approximation, such an approach results in impulsive differential equation models which are a good compromise combining  the impact of time-discrete fires with a time-continuous process of the vegetation growth. Thus they are analytically tractable,  while  at the same time remain reasonably realistic.

Now we recall the minimalistic tree-grass interactions model with pulse fires that will be used later in the paper (see Section \ref{application}). We assume that the trees and grass form an amensalistic system in which grass is harmed by the trees (which, for instance,  block the sunlight) but itself does not affect them.  The fires occur periodically every $\tilde{\tau}$ units of time. We denote by $T_n$ (resp. $G_n$) the tree (resp. grass) biomass during the inter-fire season number $ n\in\mathbb{N}^*$. Following the formalism of the recursion equations  (see Weinberger et al. \cite{Weinberger2002}, Yatat and Dumont \cite{YatatDumont2018}), the resulting minimalistic system of equations governing the trees-grass interactions with periodic fires is given by
\begin{equation}\label{IDE-recursion}
\left\{
\begin{array}{l}
\left.
\begin{array}{l}
\displaystyle\frac{d T_n}{d t}=\gamma_{T}(\textbf{W})\left(1-\displaystyle\frac{T_n}{K_{T}(\textbf{W})}\right)T_n-\delta_{T}T_n,\\
\\
\displaystyle \frac{d G_n}{d t}=\gamma_{G}(\textbf{W})\left(1-\displaystyle\frac{G_n}{K_{G}(\textbf{W})}\right)G_n-\delta_{G}G_n-\eta_{TG}T_nG_n,\\
\end{array}
\right. 0\leq t\leq \tilde{\tau},\\
\\
\left.
\begin{array}{l}
T_{n+1}(0)=(1-\psi(T_n(\tau))w_G(G_n(\tau)))T_n(\tau),\\
\\
G_{n+1}(0)=(1-\eta)G_n(\tau),\\
\end{array}
\right.
\end{array}
\right.
\end{equation}
  with non negative initial conditions $(T_1(0), G_1(0))$.

Here, between two successive fires; that is, in the inter-fire season $n$, the dynamics of both the tree and grass biomasses is modelled by the first two equations of system (\ref{IDE-recursion}), where $\gamma_G(\textbf{W})$ and
$\gamma_T(\textbf{W})$ denote the unrestricted rates of growth of the grass and the tree biomass, respectively, while
 $K_{G}(\textbf{W})$ and $K_{T}(\textbf{W})$ are the carrying capacities for grass and the trees, respectively.  All these functions are assumed to be increasing and bounded functions of the water availability $\textbf{W}$ which is supposed to be known. Further, $\delta_{G}$ and $\delta_{T}$ denote, respectively, the rates of the grass and the tree biomass loss due to natural causes,
herbivores (grazing and/or browsing) or human actions, while $\eta_{TG}$ denotes rate of the loss of the grass biomass due to the existence of trees per units of the biomasses.

At the end of the inter-fire season a fire occurs and impacts both the tree and grass biomasses. Thus there is an update of the biomasses for the beginning of the next inter-fire season. This event is modelled by the last two equations of system (\ref{IDE-recursion}). Here we assume that the fire intensity, denoted by $w_G$, is an
increasing and bounded function of the grass biomass with one as the upper bound. We also assume that $w_G(0)=0$ and $w_G'(0)=0.$
Impulsive fire-induced tree/shrub mortality, denoted by $\psi$, is assumed
to be a positive, decreasing, and nonlinear function of the tree biomass. Its upper bound is
 taken as one. Further,   $\eta$ is
the specific loss of the grass biomass due to the fire.  To avoid the extinction of either $T_n,$ or $G_n$, we assume that (see also Yatat
et al.  \cite{Yatat2018})
\begin{equation}\label{tecknical-Assumption}
\gamma_{G}(\textbf{W})-\delta_{G}>0\quad \mbox{and}\quad
\gamma_{T}(\textbf{W})-\delta_{T}>0.
\end{equation}
System (\ref{IDE-recursion}) will be studied in Section \ref{application}.

Readers are referred to Yatat
et al.  \cite{Yatat2018} for the derivation and analysis of system (\ref{IDE-recursion}) following the formalism of Lakshmikantham et al.  \cite{Lakshmikantham1989}.

\subsection{Space-explicit impulsive models}
The formulation of space-explicit impulsive models generally consists in the addition of local or non-local spatial operators to a temporal impulsive model. In Akhmet et al. \cite{Akhmet2006}, Li et al. \cite{Li2013} and Liu et al. \cite{Liu2011}, the authors considered impulsive reaction-diffusion equations to model spatio-temporal dynamics of ecological species with  prey-predator interactions and experiencing pulse and periodic perturbations like harvest, release, etc. The spatial movement of species was modelled by the Laplace operator with a constant diffusion rate. Qualitative analysis of these models was done by using the theory of sectorial operators (Henry \cite{Henry1981}, Rogovchenko \cite{Rogovchenko1997b}, \cite{Rogovchenko1997a}, Li et al. \cite{Li2013}) and comparison arguments (Rogovchenko \cite{Rogovchenko1996}, Walter \cite{Walter1997}, Liu et al. \cite{Liu2011}, Akhmet et al. \cite{Akhmet2006}). More precisely, the authors obtained some conditions involving the pulse time period that ensured the permanence of the predator-prey system and the existence of a unique globally stable periodic solution (Akhmet et al. \cite{Akhmet2006}, Li et al. \cite{Li2013}, Liu et al. \cite{Liu2011}). Vasilyeva et al. \cite{Vasilyeva2016} dealt with the question of persistence versus extinction in a single population model featuring a non-local impulsive reaction-advection-diffusion model for an insect population. The non-local term was used to  describe the dispersal of the adult insects by flight. The authors employed a dispersal kernel that gave the probability density function of the signed dispersal distances.

We note that the study of the species spread and their wave speeds, when they experience impulsive and periodic perturbations, in the case of scalar equations was done in  Vasilyeva et al. \cite{Vasilyeva2016}, Lewis and Li \cite{Lewis2012}, or Yatat and Dumont \cite{YatatDumont2018}.  However, systems of impulsive reaction-diffusion have not received much attention, (Huang et al. \cite{Huang2017}). We aim to address this question here by extending the minimalistic trees-grass interactions model (\ref{IDE-recursion}). To this end, we assume that both the woody and herbaceous plants can propagate in space through diffusion; see Yatat et al.  \cite{Yatat2017b} for a
discussion of the construction of the trees-grass interactions partial differential
equations models. The resulting minimalistic system of impulsive reaction-diffusion equations is then given by
\begin{equation}\label{I-savnna}
\left\{
\begin{array}{l}
\left.
\begin{array}{l}
\displaystyle\frac{\partial T_n}{\partial t}=d_T(\textbf{W})\displaystyle \frac{\partial^2 T_n}{\partial x^2}+\gamma_{T}(\textbf{W})\left(1-\displaystyle\frac{T_n}{K_{T}(\textbf{W})}\right)T_n-\delta_{T}T_n,\\
\\
\displaystyle \frac{\partial G_n}{\partial t}=d_G(\textbf{W})\displaystyle \frac{\partial^2 G_n}{\partial x^2}+\gamma_{G}(\textbf{W})\left(1-\displaystyle\frac{G_n}{K_{G}(\textbf{W})}\right)G_n-\delta_{G}G_n-\eta_{TG}T_nG_n,\\
\end{array}
\right.  0\leq t\leq \tilde{\tau},\quad x\in\R,\\
\\
\left.
\begin{array}{l}
T_{n+1}(x,0)=(1-\psi(T_n(x,\tau))w_G(G_n(x,\tau)))T_n(x,\tau),\\
\\
G_{n+1}(x,0)=(1-\eta)G_n(x,\tau),\\
\end{array}
\right.
\end{array}
\right.
\end{equation}
with given non negative initial conditions
\begin{equation}
(T_1(x,0), G_1(x,0)). \label{pulsed_swv_eq2}
\end{equation}
In system (\ref{I-savnna}), $d_T(\textbf{W})$ and $d_G(\textbf{W})$ denote the woody, respectively, herbaceous biomass spatial vegetative diffusion coefficient, while the remaining coefficients and assumptions on them are as in (\ref{IDE-recursion}).

The aims of the present study include proving the existence of monostable traveling wave solutions to (\ref{I-savnna}) and also the computation of the spreading speeds for them. To achieve these objectives, we use the results on monotone and monostable recursion equations and then transfer them to monotone and monostable systems of impulsive reaction-diffusion equations, as in \cite{Li2005}. We stress here that the case of bistable recursion equations is still an open problem.

\section{System of impulsive reaction-diffusion equations in unbounded domains: spreading speeds and traveling waves}\label{generic-framework}
In this section we introduce the basic notation, definition and results that allow us to deal with the issues of the
existence of traveling wave solutions and/or computation of the
spreading speeds for impulsive reaction-diffusion (IRD) systems.
 Let us denote:\\
$\textbf{u}=(u_1, u_2,...,u_N)$, $\textbf{F}=(F_1, F_2,...,F_N)$,
$\textbf{H}=(H_1, H_2,...,H_N)$, $D=diag(d_1, d_2,...,d_N)$ with
$d_i>0$, for $i=1,2,...,N$ and $\displaystyle\frac{\partial^2
\textbf{u}}{\partial x^2}:=\left(\displaystyle\frac{\partial^2
u_1}{\partial x^2},\displaystyle\frac{\partial^2 u_2}{\partial
x^2},...,\displaystyle\frac{\partial^2 u_N}{\partial x^2}\right)$.
In the case when there are no impulsive perturbations, the
reaction-diffusion system is written as
\begin{equation}\label{RD-general}
 \begin{array}{lcl}
 \displaystyle\frac{\partial \textbf{u}(x,t)}{\partial t} &=& D\displaystyle\frac{\partial^2 \textbf{u}(x,t)}{\partial x^2}+\textbf{F}(\textbf{u}(x,t)), \quad t> 0,\quad x\in\mathbb{R},\\
 \end{array}
\end{equation}
together with sufficiently smooth and nonnegative initial condition

\begin{equation}\label{RD-CI-general}
\begin{array}{lcl}
 \textbf{u}(x,0) &=& \textbf{u}_0(x), \quad x\in\mathbb{R}.\\
 \end{array}
\end{equation}

For $\tau$-periodic impulsive perturbations, we consider the whole time interval as a succession of inter-perturbation seasons of length $\tau$. Let us denote
the state variables at time $t\in[0, \tau]$ and location $x$ during
the inter-perturbation season $n\in\mathbb{N}^*$ as $\textbf{u}_n(t,
x)=(u_{n,1}, u_{n,2},...,u_{n,N})$. Following the recursion
formalism (Lewis and Li   \cite{Lewis2012}, Vasilyeva et al.
 \cite{Vasilyeva2016}, Fazly et al.   \cite{Fazly2017},
Huang et al.  \cite{Huang2017}, Yatat and Dumont
\cite{YatatDumont2018}), the impulsive reaction-diffusion
system is written as
\begin{equation}\label{IRD-general}
 \begin{array}{lcl}
 \displaystyle\frac{\partial \textbf{u}_n(x,t)}{\partial t} &=& D\displaystyle\frac{\partial^2 \textbf{u}_n(x,t)}{\partial x^2}+\textbf{F}(\textbf{u}_n(x,t)), \quad 0\leq t\leq \tau,\quad x\in\mathbb{R},\\
 \end{array}
\end{equation}
together with the updating condition
  \begin{equation}\label{I-mise-a-jour-general}
\begin{array}{lcl}
\textbf{u}_{n+1}(x,0) &=& \textbf{H}(\textbf{u}_n(x,\tau))\\
\end{array}
\end{equation}
and with sufficiently smooth and nonnegative initial data
$\textbf u_1(x,t)= \textbf{u}_{0}(x)$. We note that (\ref{I-savnna}) is a special case of (\ref{I-mise-a-jour-general}).

Our work is based on the results of Li et al.  \cite{Li2005} (see also Weinberger et al.
\cite{Weinberger2002})  concerning the existence of monostable traveling wave
solutions and the computation of spreading speeds for systems of reaction-diffusion equations. We note that they focused on  the case, when $\textbf{H}$ was just the time-$\tau$-map operator solution, $Q_\tau$, of
system (\ref{IRD-general}). Since, however, the reduction of an IRD system
 to the recursion form does not depend on the updating condition, the results of {\it op. cit.}
on the existence of traveling waves for recursions and the spreading speeds determined by them can be used verbatim to the recursion obtained from (\ref{IRD-general})-(\ref{I-mise-a-jour-general}). Thus we recall the relevant results from  \cite{Li2005}.

 We first assume
that $\textbf{F}$, $\textbf{H}$ and the initial data are such that
the IRD system (\ref{IRD-general})-(\ref{I-mise-a-jour-general})
admits a unique nonnegative classical solution for each $n \in \mathbb N^*$ (Zheng
\cite{Zheng2004}, Volpert  \cite{Volpert2014}, Logan
\cite{Logan2008},   \cite{Logan2015}, Perthame
\cite{Perthame2015}).

 We begin with some
notation (see Weinberger et al.   \cite{Weinberger2002}, Li et
al.   \cite{Li2005}). For two vector-valued functions
$\textbf{u}(x)$ and $\textbf{v}(x)$,
$\textbf{u}(x)\leq\textbf{v}(x)$ means that $u_i(x)\leq v_i(x)$ for
all $i=1,2,...,N$ and $x\in\R$, $\max\{\textbf{u}(x),
\textbf{v}(x)\}$ means the vector-valued function whose $i^{th}$
component at $x$ is $\max\{u_i(x), v_i(x)\}$, and
$\limsup\limits_{n\rightarrow\infty}\textbf{u}^{(n)}(x)$ is the
function whose $i^{th}$ component at $x$ is
$\limsup\limits_{n\rightarrow\infty}u^{(n)}(x)$. We shall, moreover,
use the usual symbol $\textbf{u}\gg \textbf{v}$ for $u_i(x)
> v_i(x)$ for all $i$ and $x$. We use the notation \textbf{0} for the constant vector whose all components are 0. If $\beta\gg\textbf{0}$ is a
constant $N-$vector, we define the set of functions
$$\mathcal{C}_\beta:=\{\textbf{u}(x):\quad \textbf{u}\quad \mbox{is continuous and}\quad \textbf{0}\leq\textbf{u}(x)\leq\beta\}.$$
Let $Q_{\tau}$ be the time-$\tau$-map operator solution of
system (\ref{IRD-general}). Then (\ref{I-mise-a-jour-general}) can be written as
\begin{equation}\label{Q-tau-operator-general}
\textbf{u}_{n+1}(x,0)=\textbf{H}(Q_{\tau}[\textbf{u}_{n}(x,0)])=:Q[\textbf{u}_{n}(x,0)], \quad n \geq 1, \quad x\in\R,
\end{equation}
with the initial condition $\textbf u_0(x)$.

In the sequel, we recall the key assumptions of Li et al.
\cite{Li2005} related to the operator $Q$ defined in (\ref{Q-tau-operator-general}).
For a fixed  $y\in \R,$ the translation operator by $y$ is defined by $T_y[\textbf u](x)=\textbf u(x-y)$, for all $x\in \R$.\smallskip\\
\textbf{Hypothesis 2.1.}
\begin{itemize}
    \item[i.] The operator $Q$ is order preserving in the sense that
    if $\textbf{u}$ and $\textbf{v}$ are any two functions in $\mathcal{C}_\beta$ with
    $\textbf{v}\geq \textbf{u}$, then $Q[\textbf{v}]\geq
    Q[\textbf{u}]$. In biological terms, the dynamics
are cooperative.
    \item[ii.]$Q[\textbf{0}]=\textbf{0}$, there is a constant vector
    $\beta\gg\textbf{0}$ such that
    $Q[\beta]=\beta$, and if $\textbf{u}_0$ is any
    constant vector with $\textbf{u}_0\gg\textbf{0}$, then the
    constant vector $\textbf{u}_n$, obtained from the recursion
    (\ref{Q-tau-operator-general}), converges to $\beta$ as $n$
approaches infinity. This hypothesis, together with (i), imply that
$Q$ takes $\mathcal{C}_\beta$ into itself, and that the equilibrium
$\beta$ attracts all initial functions in $\mathcal{C}_\beta$ with
uniformly positive components.  There may also be other
equilibria lying between $\beta$ and the extinction equilibrium
\textbf{0}, in each of which at least one of the species is extinct.
    \item[iii.]$Q$ is translation invariant. In biological terms this means
that the habitat is homogeneous, so that the growth and migration
properties are independent of location.
    \item[iv.] For any $\textbf{v}, \textbf{u}\in \mathcal{C}_\beta$ and any fixed $y$, $|Q[\textbf{v}](y)-Q[\textbf{u}](y)|$
    is arbitrarily small, provided
    $|\textbf{v}(x)-\textbf{u}(x)|$ is sufficiently small on a
    sufficiently long interval centered at $y$.
    \item[v.]Every sequence $\textbf{v}_n$ in
    $\mathcal{C}_{\beta}$ has a subsequence
    $\textbf{v}_{n_l}$ such that $Q[\textbf{v}_{n_l}]$ converges
    uniformly on every bounded set.
\end{itemize}

We are now in position to recall results of Li et al.
\cite{Li2005} that deal with spreading speeds as well as traveling
wave solutions for the IRD system
(\ref{IRD-general})-(\ref{I-mise-a-jour-general}), rewritten
following the recursion formalism (\ref{Q-tau-operator-general}). In
the sequel, we assume that \textbf{Hypothesis 2.1.} holds for the
recursion operator $Q$ of system (\ref{Q-tau-operator-general}).

Following Li et al.   \cite{Li2005} (see also Weinberger et al.
  \cite{Weinberger2002}), we consider a continuous
$\R^N$-valued function $\phi(x)$ with the properties
\begin{equation}\label{phi-properties-general}
\begin{array}{l}
i.\quad \phi(x)\quad \mbox{is non-increasing in}\quad x;\\
ii.\quad \phi(x)=\textbf{0}\quad \mbox{for all}\quad x\geq0;\\
iii.\quad \textbf{0}\ll\phi(-\infty)\ll \beta.
\end{array}
\end{equation}
We let, for all fixed $c\in\R$, $\textbf{a}_0(c;s)=\phi(s)$, and define the sequence
$\textbf{a}_n(c;s)$ by the recursion
\begin{equation}\label{an}
    \textbf{a}_{n+1}(c;s)=\max\{\phi(s),Q[\textbf{a}_n(c;x)](s+c)\}.
\end{equation}

Li et al.  \cite{Li2005} showed that the sequence
$\textbf{a}_n$ converges to a limit function $\textbf{a}(c;s)$ such that
$\textbf{a}(c;\pm\infty)$ are equilibria of $Q$ and
$\textbf{a}(c;\infty)$ is independent of the initial function
$\phi$. Following this, they defined the \emph{slowest} spreading speed
$c^*\leq\infty$ by the equation
\begin{equation}\label{c-etoile-general}
    c^*=\sup\{c: \textbf{a}(c;\infty)=\beta\}.
\end{equation}
The following result holds.

\begin{theorem}\cite[Theorem 2.1]{Li2005}\label{theorem-slowest-speed}
There is an index $j$ for which the following statement is true:
Suppose that the initial function $\textbf{u}_0(x)$ is
\textbf{\emph{0}} for all sufficiently large $x$, and that there are
positive constants $0<\rho\leq\sigma<1$ such that
$\textbf{\emph{0}}\leq\textbf{u}_0\leq\sigma \beta$ for all $x$ and
$\textbf{u}_0\geq\rho \beta$ for all sufficiently negative $x$. Then
for any positive $\varepsilon$ the solution $\textbf{u}_n$ of
recursion (\ref{Q-tau-operator-general}) has the properties

\begin{equation}\label{slowest-speed-general}
    \lim\limits_{n\rightarrow+\infty}\left[\sup\limits_{x\geq
    n(c^*+\varepsilon)}\{\textbf{u}_n\}_j(x)\right]=0
\end{equation}

and

\begin{equation}\label{slowest-speed-2-general}
    \lim\limits_{n\rightarrow+\infty}\left[\sup\limits_{x\leq
    n(c^*-\varepsilon)}\{\beta-\textbf{u}_n(x)\}\right]=\textbf{\emph{0}}.
\end{equation}
That is, the $j^{th}$ component spreads at a speed no higher than
$c^*$, and no component spreads at a lower speed.
\end{theorem}

In order to define the \emph{fastest} speed $c^*_f$, we choose
$\phi$ with the properties (\ref{phi-properties-general}), and let
$\textbf{b}_n(x)$ be the solution of the recursion
(\ref{Q-tau-operator-general}) with $\textbf{b}_0(x)=\phi(x)$.
Following Li et al. \cite{Li2005}, we define the function
$$\textbf{B}(c;x)=\limsup\limits_{n\rightarrow+\infty}\textbf{b}_n(x+nc).$$

Li et al.  \cite{Li2005} showed that $\textbf{B}(c;\infty)$ is
independent of the choice of the initial function $\phi$ as long as
$\phi$ has the properties (\ref{phi-properties-general}). We
therefore can define the \emph{fastest} spreading speed $c^*_f\geq c^*$
by the formula
\begin{equation}\label{c-etoile-f-general}
    c^*_f=\sup\{c: \textbf{B}(c;\infty)\neq \textbf{0}\}.
\end{equation}
The following result holds.
\begin{theorem} \cite[Theorem 2.2]{Li2005}\label{theorem-fastest-speed}
There is an index $i$ for which the following statement is true:
Suppose that the initial function $\textbf{u}_0(x)$ is
\textbf{\emph{0}} for all sufficiently large $x$, and that there are
positive constants $0<\rho\leq\sigma<1$ such that
$\textbf{\emph{0}}\leq\textbf{u}_0\leq\sigma \beta$ for all $x$ and
$\textbf{u}_0\geq\rho \beta$ for all sufficiently negative $x$. Then
for any positive $\varepsilon$ the solution $\textbf{u}_n$ of the
recursion (\ref{Q-tau-operator-general}) has the properties

\begin{equation}\label{fastest-speed-general}
    \limsup\limits_{n\rightarrow+\infty}\left[\inf\limits_{x\leq
    n(c^*_f-\varepsilon)}\{\textbf{u}_n\}_i(x)\right]>0
\end{equation}

and

\begin{equation}\label{fastest-speed-2-general}
    \lim\limits_{n\rightarrow+\infty}\left[\sup\limits_{x\geq
    n(c^*_f+\varepsilon)}\textbf{u}_n(x)\right]=\textbf{\emph{0}}.
\end{equation}
That is, the $i^{th}$ component spreads at a speed no less than
$c^*_f$, and no component spreads at a higher speed.
\end{theorem}

Let $\hat\textbf{u}_n, n\geq 0$ be the solution to the recursion
\begin{equation}\label{Q-tau-operator-general1}
\hat\textbf{u}_{n+1}(x)= Q[\hat \textbf{u}_{n}(x)],\quad n\geq 1,\quad x\in\R,
\end{equation}
with $\hat \textbf{u}_{0}(x) = \textbf u_0(x)$. Recall that a traveling wave of speed $c$ is a solution of the
recursion (\ref{Q-tau-operator-general}) which has the form
$\textbf{u}_n(x,0)=\textbf{Z}(x-nc)$ with $\textbf{Z}(s)$ a function
in $\mathcal{C}_{\beta}$; that is, the solution at time $n + 1$ is
simply the translate by $c$ of its value at $n$. Then such a travelling wave defines a traveling wave solution for (\ref{IRD-general})-(\ref{I-mise-a-jour-general}) in the following sense. By (\ref{Q-tau-operator-general}) we have
$$
\textbf u_n(x,0) = \hat{\textbf u}_n(x) = \textbf{Z}(x-nc) =\textbf u_0(x-nc)
$$
and thus, by \textbf{Hypothesis 2.1.} iii., for $n\tau \leq t<(n+1)\tau$ we have $\textbf u_{n+1}(x,t) = \textbf u_{1}(x-nc,t)$. We observe that since the model is translation invariant, we obtain a travelling wave for the system (\ref{IRD-general}) without the updating conditions. 

Using the definition of $c^*$ and $c^*_f$, we have the following
result that deals with the existence of traveling wave solutions for the
IRD systems (\ref{IRD-general})-(\ref{I-mise-a-jour-general}).
\begin{theorem}\cite[Theorem 3.1]{Li2005}\label{theorem-TW}
If $c\geq c^*$, then there is a non-increasing traveling wave solution
$\textbf{Z}(x-nc)$ of speed $c$ with $\textbf{Z}(-\infty)=\beta$ and
$\textbf{Z}(+\infty)$ an equilibrium other than $\beta$.

If there is a traveling wave $\textbf{Z}(x-nc)$ with
$\textbf{Z}(-\infty)=\beta$ such that for at least one component $i$
$$\liminf\limits_{x\rightarrow\infty}Z_i(x)=0,$$ then $c\geq c^*$.
If this property is valid for all components of $\textbf{Z}$, then
$c\geq c^*_f$.
\end{theorem}

In practice, assumptions iii., iv. and v. are typically satisfied
for biologically reasonable (impulsive) models. The most challenging
 assumptions  for IRD systems
(\ref{IRD-general})-(\ref{I-mise-a-jour-general}) are i. and ii.

\section{Application to a minimalistic trees-grass interactions IRD system}\label{application}

In this section, we consider the minimalistic tree-grass
interactions IRD system (\ref{I-savnna})-(\ref{pulsed_swv_eq2}).
Using a similar normalization procedure as in Yatat et al.
\cite{Yatat2017b} (see also Appendix \ref{Appendix-scaling}), system
(\ref{I-savnna})-(\ref{pulsed_swv_eq2}) becomes

\begin{equation}\label{cooperation0}
 \left\{%
\begin{array}{lcl}
 \displaystyle\frac{\partial U_n}{\partial t} &=& U_n(1-U_n)+d_u\displaystyle\frac{\partial^2 U_n}{\partial x^2}, \quad 0\leq t\leq \tau,\quad x\in\mathbb{R},\\
 % & & \\
    \displaystyle\frac{\partial V_n}{\partial t} &=&
  \lambda V_n(1-V_n- \gamma U_n)+d_v\displaystyle\frac{\partial^2 V_n}{\partial
  x^2},
\end{array}
\right.
\end{equation}
together with the updating conditions
  \begin{equation}\label{mise-a-jour0}
 \left\{%
\begin{array}{lcl}
U_{n+1}(x,0) &=& (1- w_V(V_{n}(x,\tau))\psi(U_n(x,\tau)))U_n(x,\tau),\\
    V_{n+1}(x,0) &=& (1- \eta)V_{n}(x,\tau),\\
\end{array}
\right.
\end{equation}
and sufficiently smooth and nonnegative initial data
$U_{1}(x,0), V_{1}(x,0)$. We are now looking for
traveling wave solutions as well as the spreading speeds involving semi-trivial equilibria; that is, equilibria where either
$T_n=0,$ or $G_n=0$ but not simultaneously $T_n=0 =G_n=0.$

\subsection{Basic properties of (\ref{cooperation0})-(\ref{mise-a-jour0})}
Let $C_{ub}(\R)$ be the Banach space of bounded, uniformly
continuous function on $\R$ and $$C^{2}_b(\R):=\{f\in C_{ub}(\R): f''\in
C_{ub}(\R)\}.$$ $C_{ub}(\R)$ and $C^{2}_b(\R)$ are endowed with the
following (sup) norms
\begin{equation}\label{borme-C-ub}
\|f\|_{C_{ub}(\R)}=\|f\|_{\infty}=\sup\limits_{x\in\R}|f(x)|
\end{equation}
and
\begin{equation}\label{norme-C-b}
\|f\|_{C_{b}^2(\R)}=\|f\|_{C_{ub}(\R)}+\|f''\|_{C_{ub}(\R)}.
\end{equation}
$C_{b}^2(\R)$ endowed with the norm $\|\cdot\|_{C_{b}^2(\R)}$ is a
Banach space.

We recall that we assumed that $w_V$ was an increasing $\mathcal{C}^1(\mathbb R)$
function such that  for all $V\in \R$,
\begin{equation}\label{omegaV-properties}
w_V(0)=0,\quad w_V'(0)=0,\quad 0\leq w_V(V)<1.
\end{equation}
Similarly,  $\psi$ is a decreasing $\mathcal{C}^1(\R)$
function such that  for all $U\in C_{ub}(\R)$,
\begin{equation}\label{psi-properties}
\psi(0)>0,\quad \psi'(0)<0,\quad 0< \psi(U)\leq1.
\end{equation}
 For simplicity we note $(U_{n}(x,0),
V_{n}(x,0))=(U_{n,0}(x),V_{n,0}(x)).$
 In the sequel, we first
address the question of the existence and uniqueness of solutions of the
reaction-diffusion (RD) system (\ref{cooperation0}) in unbounded
domains.

For fixed $n\in \mathbb{N}^*$, we set $\mathbf w=(w_1,w_2) :=(U_n,V_n)$.  System  (\ref{cooperation0}) can be written as the abstract Cauchy problem
\begin{equation}\label{Abstract-form}
\left\{
 \begin{array}{l}
 \displaystyle\frac{d \mathbf w}{d t} +A \mathbf w= \mathbf F(\mathbf w),\\
 \mathbf w(0)=\mathbf  w_0,
 \end{array}
 \right.
\end{equation}
where in the Banach space $B=C_{ub}(\R)\times C_{ub}(\R)$ we have
\begin{equation}%\label{RD-A}
\left\{
\begin{array}{l}
  D(A) = C^{2}_b(\R)\times C^{2}_b(\R), \\
  a=diag(d_u,d_v),\\
  A\mathbf w = -a\mathbf w'',\\
  \mathbf F: D(A) \rightarrow D(A), \mathbf F(\mathbf w)=(U_n(1-U_n),\lambda V_n(1-V_n-\gamma U_n)).
\end{array}
\right.
\end{equation}
For $X\in\{C_{ub}(\R),C^{2}_b(\R)\}$ and $(a,b)\in X\times X$ we define
$$\|(a,b)\|_{X\times X}=\|a\|_{X}+\|b\|_{X}.$$

We shall consider (\ref{Abstract-form}) as a nonlinear perturbation of the linear part that, in this case, consists of two uncoupled diffusion equations. Thus, the corresponding semigroup is the diagonal semigroup consisting of Gauss semigroups
\begin{equation}
\mathbf {S}(t) \mathbf w = diag (G_{d_u}(t)\star w_1, G_{d_v}(t)\star w_2), \quad t>0, \qquad \mathbf  S(0)\mathbf w =\mathbf  w,
\label{Gsem1}
\end{equation}
where for $d=d_u,d_v$  and $f\in X$
\begin{equation}\label{Gauss-semigroup}
\begin{array}{ccl}
  (S_d(t)f)(x) &=& [G_{d}\star f](x,t) = \displaystyle\frac{1}{\sqrt{4\pi d
t}}\displaystyle\int_{\R}\exp\left(-\displaystyle\frac{(x-y)^2}{{4d
t}}\right)f(y)dy,   \\
\end{array}
\end{equation}
where $$G_d(x,t)=\displaystyle\frac{1}{\sqrt{4\pi d
t}}\exp\left(-\displaystyle\frac{x^2}{{4d t}}\right).$$
Then, by e.g. \cite[Section 7.3.10]{Bob}, the family $\{\mathbf  S(t)\}_{t\geq0}$ is a $C_0-$semigroup of contractions (even analytic) on
$B,$ with the generator $(A,D(A))$. Furthermore, since $\mathbf  F$ is a quadratic function, it is continuously Fr\'{e}chet differentiable in $B$ and therefore (\ref{Abstract-form}) has a unique local in time (defined on $[0,t_{max})$) classical solution, provided $\mathbf w(0)\in D(A)$ (due to the analyticity, there is a local classical solution with $\mathbf w(0)\in B$ on $(0,t_{max})$, see e.g. \cite[Theorem 2.3.5]{Zheng2004}).

Our problem is posed on the whole line and thus comparison theorems for the solutions are a little more delicate. Though in various forms they appear in many papers, see e.g. \cite{term,Fife1979,Paobook} and references therein, and thus it seems that they belong to a mathematical folklore, a comprehensive proof of them, starting from the first principles, is difficult to find. Therefore we decided to provide a such a proof for the problem at hand that uses the positivity of the semigroup $\{S_d(t)\}_{t\geq 0}$ and the triangular  structure of the nonlinearity in (\ref{cooperation0}). In fact, the semigroup for  the scalar problem,
\begin{eqnarray}
\phi_t &=& d\phi_{xx} + c(x,t)\phi,\quad x \in \mathbb R,\quad t\geq 0,\nonumber\\
\phi(x,0)&=&\mathring{\phi}(x),
\label{C50}
\end{eqnarray}
where   $|c(x,t)|\leq L$ on $\mathbb R\times \mathbb R_+$, is positive.  Indeed,  the equation can be re-written as
\begin{eqnarray}
\Phi_t&= &d\Phi_{xx} +C(x,t) \Phi,\quad x \in \mathbb R, \quad t\geq 0,\nonumber\\
\Phi(x,0)&=& \mathring \phi(x),
\label{C5}
\end{eqnarray}
where $C(x,t)=c(x,t)+L\geq 0$ and $\Phi(x,t)=e^{Lt}\phi(x,t)$ and the positivity of the semigroup solving (\ref{C5}) follows from the Dyson-Phillips expansion \cite[Theorem III.1.10]{Engel2006}.
Then, considering two solutions $u_1$ and $u_2$ with $u_1(x,0)\leq u_2(x,0)$ to the scalar nonlinear problem
\begin{equation}
u_t = d u_{xx} + F(u,t),\quad x \in \mathbb R, \quad t>0,
\label{C1}
\end{equation}
on a common interval of existence $[0,t']$, where $F$ is a differentiable function on $\mathbb R\times \mathbb R_+$, we find that $z= u_2-u_1$ satisfies
\begin{equation}
z_t = d z_{xx} + c(x,t) z, \quad x \in \mathbb R,\quad 0\leq t\leq t'
\label{C8}
\end{equation}
where $c = F'((1-\theta)u_1 + \theta u_2), 0< \theta< 1,$ is bounded on $\mathbb R\times [0,t']$. By the above linear result, $u_2 -u_1 =  z\geq 0$. Returning now to (\ref{cooperation0}), we see that the first equation is the Fisher equation and functions identically equal to $0$ and to $1$ are its solutions defined globally in time. Thus for any $0\leq U(x,0)\leq 1$ we obtain $0\leq U(x,t)\leq 1$ on $[0,t_{max})$. Hence $U$ is defined globally in $t$ and satisfies  $0\leq U(x,t)\leq 1$ for all $(x,t)\in \mathbb R\times \mathbb R_+$. Now, let $V$ be the solution of the second equation in (\ref{cooperation0}) on the maximum interval of existence $[0,t_{max})$,
$$
V_t = d_v V_{xx} + V(1-U(x,t) -V),
$$
with $0\leq V(x,0)\leq 1$. Since the function identically equal to zero solves the above equation, as before we get $V(x,t)\geq 0$ on $[0,t_{max})$ as long as $V(x,0)\geq 0$. But then, using $V(1-U(x,t) -V) \leq V(1-V)$ on account of $U\geq 0$ we see, by e.g. Picard iterates, that $V$ is dominated by the solution of the Fisher equation with the same initial condition and so, in particular, by $1$. This gives the global in time existence of $V$ and the bound $0\leq V\leq 1$,  and hence global in time solvability of the system (\ref{cooperation0}) with initial conditions bounded by 0 and 1.

Since the updating conditions (\ref{mise-a-jour0}) are non-increasing, if the initial data $U_{1}(\cdot,0),
V_{1}(\cdot,0)$ satisfy
\begin{equation}\label{bornage2}
\|U_1(\cdot,0)\|_\infty\leq 1 \quad \mbox{and} \quad
\|V_1(\cdot,0)\|_\infty\leq 1,
\end{equation}
then for each $n\in
\mathbb{N}^*$,  the solutions $(U_n,V_n)$ of system
(\ref{Abstract-form}) satisfy
\begin{equation}\label{bornage1}
\|U_n\|_\infty\leq 1\quad \mbox{and}\quad
\|V_n\|_\infty\leq 1.
\end{equation}

\subsection{The existence of equlibria of (\ref{cooperation0})-(\ref{mise-a-jour0})}

\subsubsection{The first coordinate change}
System (\ref{cooperation0}) is monotone competitive and system
(\ref{mise-a-jour0}) is not monotone. Therefore, the full system
(\ref{cooperation0})-(\ref{mise-a-jour0}) is not monotone. Hence,
in order to be able to apply results of Li et al. \cite{Li2005}, we first proceed to a coordinates change in order to obtain a monotone
cooperative system. We set
\begin{equation}\label{new-variables1}
\left\{
    \begin{array}{ccl}
      u_n & = & U_n, \\
      v_n & = & 1-V_n
    \end{array}
\right.
\end{equation}
so that system (\ref{cooperation0})-(\ref{mise-a-jour0}) is transformed to

\begin{equation}\label{cooperation1}
 \left\{%
\begin{array}{lcl}
 \displaystyle\frac{\partial u_n}{\partial t} &=& u_n(1-u_n)+d_u\displaystyle\frac{\partial^2 u_n}{\partial x^2}, \quad 0\leq t\leq \tau,\quad x\in\mathbb{R},\\
 % & & \\
    \displaystyle\frac{\partial v_n}{\partial t} &=&
  - \lambda v_n(1-v_n) + \lambda\gamma u_n(1-v_n)+d_v\displaystyle\frac{\partial^2 v_n}{\partial
  x^2},
\end{array}
\right.
\end{equation}
together with the updating conditions
 \begin{equation}\label{mise-a-jour1}
 \left\{%
\begin{array}{lcl}
u_{n+1}(x,0) &=& (1- w_v(v_{n}(x,\tau))\psi(u_n(x,\tau)))u_n(x,\tau),\\%=:H_1(u_n(x,\tau), v_n(x,\tau)),\\
 % & & \\
    v_{n+1}(x,0) &=& (1- \eta)v_{n}(x,\tau)+\eta.%=:H_2(u_n(x,\tau),
 %   v_n(x,\tau)).
\end{array}
\right.
\end{equation}
Properties (\ref{bornage1}), (\ref{bornage2}) and
(\ref{omegaV-properties}) imply  that $w_v$ is a decreasing
$\mathcal{C}^1(\R)$ function such that
\begin{equation}\label{omegav-properties}
w_v(1)=0,\quad w_v'(1)=0,\quad 0\leq w_v<1.
\end{equation}
We also deduce that system (\ref{cooperation1}) is monotone
cooperative and the sequence defined in (\ref{mise-a-jour1}) is monotone increasing.  Hence system (\ref{cooperation1})-(\ref{mise-a-jour1}) is  monotone
 cooperative as long as the initial conditions belong to $[0,1]$.

\subsubsection{Space implicit model}
 Space homogeneous solutions of
system (\ref{cooperation1})-(\ref{mise-a-jour1}) satisfy

\begin{equation}\label{homogeneous-cooperation}
 \left\{%
\begin{array}{lcl}
 \displaystyle\frac{d u_{n+1}}{d t} &=& u_{n+1}(1-u_{n+1}), \quad 0\leq t\leq \tau,\quad n\in\mathbb{N},\\
 % & & \\
    \displaystyle\frac{d v_{n+1}}{d t} &=&
  - \lambda v_{n+1}(1-v_{n+1})+ \lambda\gamma(1-v_{n+1})u_{n+1},
\end{array}
\right.
\end{equation}
together with the updating conditions
 \begin{equation}\label{homogemeous-mise-a-jour}
 \left\{%
\begin{array}{lcl}
u_{n+1}(0) &=& (1- w_v(v_{n}(\tau))\psi(u_n(\tau)))u_n(\tau),\\
v_{n+1}(0) &=& (1- \eta)v_{n}(\tau)+\eta.
\end{array}
\right.
\end{equation}
Solving the logistic equation (\ref{homogeneous-cooperation})$_1$, leads to

\begin{equation}\label{homogemeous-u}
 u_{n+1}(t) =
 \displaystyle\frac{u_{n+1}(0)}{u_{n+1}(0)+(1-u_{n+1}(0))e^{-t}},\quad 0\leq t\leq
 \tau.
\end{equation}
In addition, direct computations give
\begin{equation}\label{Iu}
\begin{array}{lcl}
\displaystyle\int_0^tu_{n+1}(s)ds &=& \displaystyle\int_0^t\frac{u_{n+1}(0)}{u_{n+1}(0)+(1-u_{n+1}(0))e^{-s}}ds\\
 &=&
 \displaystyle\int_0^t\frac{u_{n+1}(0)e^{s}}{u_{n+1}(0)e^{s}+(1-u_{n+1}(0))}ds\\
 &=&\ln \left(1+u_{n+1}(0)(e^{t}-1)\right)\\
 &=:&\ln I_u(t).
\end{array}
\end{equation}
Now, returning to  (\ref{homogeneous-cooperation})$_2$ and setting $z=1/(1-v_{n+1})$, we get
$$
  \dot{z}  =  \lambda(1-z+\gamma u_{n+1}z) =  \lambda-\lambda(1-\gamma u_{n+1})z.
$$
Using the integrating factor $e^{\lambda  \int_0^t(1-\gamma u_{n+1}(s))ds } = e^{\lambda t} [I_u(t)]^{-\lambda \gamma}$, we get
$$
z(t) = e^{-\lambda t} [I_u(t)]^{\lambda \gamma} \frac{1}{1-v_{n+1}(0)}  + \lambda e^{-\lambda t} [I_u(t)]^{\lambda \gamma} \int_0^t e^{\lambda s} [I_u(s)]^{-\lambda \gamma}ds
$$
so that
$$v_{n+1}(t)=1-\displaystyle\frac{(1-v_{n+1}(0))e^{\lambda t}[I_u(t)]^{-\lambda \gamma}}{1+\lambda (1-v_{n+1}(0))\displaystyle\int_0^te^{\lambda s}[I_u(s)]^{-\lambda \gamma}ds}.$$
Using the updating condition (\ref{homogemeous-mise-a-jour}) leads
to

\begin{equation}\label{homogeneous-recursion-o}
\left\{
    \begin{array}{lcl}
      u_{n+1}(\tau) & = & \displaystyle\frac{(1-w_v(v_n(\tau))\psi(u_n(\tau)))u_{n}(\tau)}{e^{-\tau}+(1-e^{-\tau})(1-w_v(v_n(\tau))\psi(u_n(\tau)))u_{n}(\tau)}=:\bar{F}_1(u_n(\tau),v_n(\tau)), \\
      &&\\
       v_{n+1}(\tau) &=& 1-\displaystyle\frac{(1-\eta)(1-v_{n}(\tau)))e^{\lambda \tau}[I_u(\tau)]^{-\lambda \gamma}}{1+(1-\eta)(1-v_{n}(\tau))\lambda\displaystyle\int_0^\tau e^{\lambda s}[I_u(s)]^{-\lambda \gamma}ds}=:\bar{F}_2(u_n(\tau),v_n(\tau)).
    \end{array}
    \right.
\end{equation}

Thus, the solution of system
(\ref{homogeneous-cooperation})-(\ref{homogemeous-mise-a-jour})
given by system (\ref{homogeneous-recursion-o}) generates a discrete
dynamical system. Space homogeneous equilibria of system
(\ref{cooperation1})-(\ref{mise-a-jour1}) are equilibria of model
(\ref{homogeneous-recursion-o}).

\subsubsection{Space homogeneous equilibria of system
(\ref{cooperation1})-(\ref{mise-a-jour1})} In this section we
compute the space homogeneous equilibria of system
(\ref{cooperation1})-(\ref{mise-a-jour1}) by solving the fixed point
problem associated to system (\ref{homogeneous-recursion-o}).
\par
$\bar{F}_1(u,v)=u$ implies $u=0$ or
\begin{equation}\label{equilibre-u}
1-w_v(v)\psi(u)=e^{-\tau}+(1-e^{-\tau})(1-w_v(v)\psi(u))u.
\end{equation}
Similarly, $\bar{F}_2(u,v)=v$ implies $v=1$ or
\begin{equation}\label{equilibre-v}
1+(1-\eta)(1-v)\lambda\displaystyle\int_0^\tau e^{\lambda s}[I_u(s)]^{-\lambda \gamma}ds=(1-\eta)e^{-\lambda \tau}[I_u(\tau)]^{\lambda \gamma}.
\end{equation}
We therefore deduce the first equilibrium $\mathbf E_0=(0,1)$. Substituting
$u=0$ (i.e. $I_0(t)=1$) in (\ref{equilibre-v}) implies
$$\bar{v}=\displaystyle\frac{\eta}{(1-\eta)(e^{\lambda\tau}-1)}>0.$$ Note that  $$\bar{v}<1\quad \mathrm{if\;and\;only\;if} \quad \mathcal{R}_0>1$$ where
$\mathcal{R}_0:=(1-\eta)e^{\lambda\tau}$. Substituting $v=1$ in
(\ref{equilibre-u}) implies $u=1$. Hence, we obtain the following
Lemma \ref{homogeneous-trivial-equilibria}.
\begin{lemma}\label{homogeneous-trivial-equilibria}
 System (\ref{cooperation1})-(\ref{mise-a-jour1}) admits as trivial
 and semi-trivial equilibria in the feasible region:
\begin{itemize}
    \item $\mathbf E_0=(0,1)$ and $\mathbf E_u=(1,1)$ that always exist;
    \item
    $\mathbf E_v=(0,\bar{v})=\left(0,\displaystyle\frac{\eta}{(1-\eta)(e^{\lambda\tau}-1)}\right)$ if and only if $$\mathcal{R}_0=(1-\eta)\exp(\lambda\tau)>1.$$
\end{itemize}
\end{lemma}
In this study, we are mainly concerned with existence of traveling wave solutions of system (\ref{cooperation1})-(\ref{mise-a-jour1})
involving equilibria $\mathbf E_u$ and $\mathbf E_v$ computed in Lemma
\ref{homogeneous-trivial-equilibria}. In order to use the results of
Li et al.  \cite{Li2005}, we translate the equilibrium $\mathbf E_v$
to \textbf{0} through another coordinates change.

\subsection{The existence of travelling waves}
\subsubsection{The second coordinate change}

Recall that
$$\bar{v}=\displaystyle\frac{\eta}{(1-\eta)(e^{\lambda\tau}-1)}$$
and $$\mathcal{R}_0=(1-\eta)e^{\lambda\tau}.$$ Recall also that
$0\leq \bar{v}<1$ is equivalent to $\mathcal{R}_0>1.$ Therefore, in
this section we assume that
$$\mathcal{R}_0>1.$$ We set
\begin{equation}\label{new-variables2}
\left\{
    \begin{array}{ccl}
      u_n & = & u_n, \\
      q_n & = & v_n-\bar{v}.
    \end{array}
\right.
\end{equation}
Hence, system (\ref{cooperation1})-(\ref{mise-a-jour1}) becomes
\begin{equation}\label{cooperation}
 \left\{%
\begin{array}{lcl}
 \displaystyle\frac{\partial u_n}{\partial t} &=& u_n(1-u_n)+d_u\displaystyle\frac{\partial^2 u_n}{\partial x^2}, \quad 0\leq t\leq \tau,\quad x\in\mathbb{R},\\
 % & & \\
    \displaystyle\frac{\partial q_n}{\partial t} &=&
  - \lambda (q_n+\bar{v})(1-q_n-\bar{v}) + \lambda\gamma u_n(1-q_n-\bar{v})+d_v\displaystyle\frac{\partial^2 v_n}{\partial
  x^2},
\end{array}
\right.
\end{equation}
together with the updating conditions
 \begin{equation}\label{mise-a-jour}
 \left\{%
\begin{array}{lcl}
u_{n+1}(x,0) &=& (1- w(q_{n}(x,\tau))\psi(u_n(x,\tau)))u_n(x,\tau)=:H_1(u_n(x,\tau), q_n(x,\tau)),\\
 % & & \\
    q_{n+1}(x,0) &=& (1- \eta)q_{n}(x,\tau)+\eta(1-\bar{v})=:H_2(u_n(x,\tau),
    q_n(x,\tau)).
\end{array}
\right.
\end{equation}

As previously, we deduce from properties (\ref{omegav-properties})
that $w$ is a decreasing $\mathcal{C}^1(\R)$ function such
that
\begin{equation}\label{omegaq-properties}
w(0)>0,\quad w(1-\bar{v})=0,\quad w'(1-\bar{v})=0,\quad 0\leq
w(q)<1.
\end{equation}

For simplicity, we set $(u_{n,0},q_{n,0})=(u_{n}(x,0),q_{n}(x,0))$.
We also set $\mathbf P_{n,0}=(u_{n,0},q_{n,0})$ and let $Q_{\tau}$ denote the
time-$\tau$-map operator solution of system (\ref{cooperation}). Then
\begin{equation}\label{Q-tau-operator}
\mathbf P_{n+1,0}=\mathbf H(Q_{\tau}[\mathbf P_{n,0}])=:Q[\mathbf P_{n,0}]
\end{equation}
where $\mathbf H(u,q)=(H_1(u,q), H_2(u,q))$.

 Using the coordinates change (\ref{new-variables2}),  system
(\ref{homogeneous-recursion-o}) becomes
\begin{equation}\label{homogeneous-recursion}
\left\{
    \begin{array}{lcl}
      u_{n+1}(\tau) & = & \displaystyle\frac{(1-w(q_n(\tau))\psi(u_n(\tau)))u_{n}(\tau)}{e^{-\tau}+(1-e^{-\tau})(1-w(q_n(\tau))\psi(u_n(\tau)))u_{n}(\tau)}=:F_1(u_{n}(\tau),q_{n}(\tau)), \\
      &&\\
       q_{n+1}(\tau) &=& 1-\bar{v}-\displaystyle\frac{(1-\eta)(1-\bar{v}-q_{n}(\tau))e^{\lambda \tau}[I_u(\tau)]^{-\lambda \gamma}}{1+\lambda(1-\eta)(1-\bar{v}-q_{n}(\tau))\displaystyle\int_0^\tau e^{\lambda s}[I_u(s)]^{-\lambda \gamma}ds}=:F_2(u_{n}(\tau),v_{n}(\tau))
    \end{array}
    \right.
\end{equation}
and, using  Lemma \ref{homogeneous-trivial-equilibria},
it is straightforward to deduce that %the following result
%\begin{lemma}\label{homogeneous-trivial-equilibria-2}
 system (\ref{cooperation})-(\ref{mise-a-jour}) admits as space homogeneous equilibria:
 $$\mathbf e_0=(0,1-\bar{v}),\quad \mathbf e_u=(1,1-\bar{v})\quad \mbox{and}\quad \mathbf  e_v=(0,0).$$
%\end{lemma}

\subsubsection{Stability analysis of space homogeneous equilibria of system
(\ref{cooperation})-(\ref{mise-a-jour})} We first focus on the
integral term that appears in $F_2$ (see equation
(\ref{homogeneous-recursion})$_2$). Recalling (\ref{Iu}), we consider
$$Y(u,q):=\int_0^\tau e^{\lambda
s}\left(1+u(e^s-1)(1-w(q)\psi(u))\right)^{-\lambda \gamma}ds$$
 and an auxiliary function $Z: (u,q,s)\in\mathbb{R}^+\times\mathbb{R}^+\times[0,
\tau]\rightarrow\mathbb{R}$ defined by $$Z(u,q,s)=e^{\lambda
s}\left(1+u(e^s-1)(1-w(q)\psi(u))\right)^{-\lambda \gamma}.$$ For
every $u,q\in\mathbb{R}^+$, the function $s\mapsto Z(u,q,s)$ is
continuous on the interval $[0,\tau]$.
 In addition,
$$\displaystyle\frac{\partial Z}{\partial u}=-\lambda\gamma e^{\lambda s}(e^s-1)\left(1-w(q)(\psi(u)-u\psi'(u))\right)
\left(1+u(e^s-1)(1-w(q)\psi(u))\right)^{-\lambda\gamma-1}$$
exists and is continuous for all $(u,q,s)\in\mathbb{R}^+\times\mathbb{R}^+\times[0,
\tau]$. Consequently,
$\displaystyle\frac{\partial Y}{\partial u}=\displaystyle\int_0^\tau\frac{\partial Z}{\partial u}ds$. Similarly, $\displaystyle\frac{\partial Y}{\partial q}=\displaystyle\int_0^\tau\frac{\partial
Z}{\partial q}ds$. For convenience, we set (see (\ref{homogeneous-recursion}))
$$F_1(u,q)=\displaystyle\frac{A_1(u,q)}{A_2(u,q)} \quad\mbox{and} \quad F_2(u,q)=1-\bar{v}+\displaystyle\frac{B_1(u,q)}{B_2(u,q)}.$$
Computing the partial derivatives of $A_1, A_2,B_1$ and $B_2$
defined in (\ref{homogeneous-recursion}) gives
\begin{equation}\label{jacobien-1}
\begin{array}{lcl}
  \displaystyle\frac{\partial A_1}{\partial u} & = & 1-w(q)(\psi(u)+u\psi'(u)), \\
  \displaystyle\frac{\partial A_1}{\partial q} & = & -w'(q)\psi(u)u,\\
 \displaystyle \frac{\partial A_2}{\partial u} & = & (1-e^{-\tau})\displaystyle \frac{A_1}{\partial u},\\
 \displaystyle \frac{\partial A_2}{\partial q} & = & (1-e^{-\tau})\displaystyle \frac{A_1}{\partial q},\\
 \displaystyle\frac{\partial B_1}{\partial u} & = & \gamma\lambda(1-\eta)(1-q-\bar{v})e^{\lambda\tau}(1+(e^\tau-1)u(1-w(q)\psi(u)))^{-\gamma\lambda-1}
 (e^\tau-1)(1-w(q)(\psi(u)+u\psi'(u))),\\
 \displaystyle\frac{\partial B_1}{\partial q} & = & (1-\eta)e^{\lambda\tau}(1+(e^\tau-1)u(1-w(q)\psi(u)))^{-\gamma\lambda}
 \\
 &-&\gamma\lambda(1-\eta)(1-q-\bar{v})e^{\lambda\tau}(e^\tau-1)uw'(q)\psi(u)(1+(e^\tau-1)u(1-w(q)\psi(u)))^{-\gamma\lambda-1},\\
  \displaystyle\frac{\partial B_2}{\partial u} & = &-\gamma\lambda^2 (1-\eta)(1-q-\bar{v})\\
  &\times&\displaystyle\int_0^\tau e^{\lambda s}(e^s-1)(1-w(q)(\psi(u)+u\psi'(u))(1+(e^s-1)u(1-w(q)\psi(u)))^{-\gamma\lambda-1}ds,\\
 \displaystyle\frac{\partial B_2}{\partial q} & = &-(1-\eta)\lambda\displaystyle\int_0^\tau e^{\lambda s}(1+(e^s-1)u(1-w(q)\psi(u)))^{-\gamma\lambda}ds\\
  &+&(1-\eta)(1-q-\bar{v})\gamma\lambda^2\displaystyle\int_0^\tau e^{\lambda s}(e^s-1)uw'(q)\psi(u) (1+(e^s-1)u(1-w(q)\psi(u)))^{-\gamma\lambda-1}ds.\\
 \end{array}
\end{equation}
Let  $\mathcal J= \{J_{ij}\}_{1\leq i,j\leq 2}$ denote the Jacobian matrix of (\ref{homogeneous-recursion}).  Using (\ref{jacobien-1}), the quotient rule and the properties of $w$ (see (\ref{omegaq-properties})) and  $\psi$ (see
(\ref{psi-properties})), we obtain the
following results:
\begin{itemize}
    \item Local stability of $\mathbf e_0$.  At the
    trivial equilibrium $\mathbf e_0=(0,1-\bar{v})$ the matrix $\mathcal J$ has the following
    entries:
    $$
    \begin{array}{ccl}
      J_{11} & = & e^\tau, \\
      J_{12} & = & 0, \\
      J_{21} & = & 0, \\
      J_{22} & = & (1-\eta)e^{\lambda\tau}. \\
    \end{array}
    $$ Since $e^\tau$ is an eigenvalue of $\mathcal J$  at $\mathbf e_0$ and $e^\tau>1,$ the equilibrium $\mathbf e_0$ is unstable.
    \item Local stability of $\mathbf e_u$. At the
    semi-trivial equilibrium $\mathbf e_u=(1,1-\bar{v})$  the matrix $\mathcal J$ has the following entries:
    $$
    \begin{array}{ccl}
      J_{11} & = & e^{-\tau}, \\
      J_{12} & = & 0,\\
      J_{21} & = & 0, \\
      J_{22} & = & (1-\eta)e^{\lambda\tau(1-\gamma)}=:\mathcal{R}_1. \\
    \end{array}$$
Eigenvalues of the Jacobian matrix at $\mathbf e_u$ are
$e^{-\tau}$ and $\mathcal{R}_1$ with $e^{-\tau}<1$.
    Therefore, $\mathbf e_u$ is locally asymptotically stable (LAS) whenever
    $\mathcal{R}_1<1$.
    \item Local stability of $\mathbf e_v$. $\mathcal J$  at the
    semi-trivial equilibrium $\mathbf e_v=(0,0)$ has the following entries:
  $$  \begin{array}{ccl}
      J_{11} & = & \left(1-w(0)\psi(0)\right)e^\tau, \\
      J_{12} & = & 0, \\
%      J_{21} & \in &\mathbb{R},\\
      J_{22} & = & \displaystyle\frac{1}{(\mathcal{R}_0)^2}.\\
    \end{array}$$
    Knowing the explicit value of $J_{21}$ is not necessary since $J_{12}=0$.
Eigenvalues of the Jacobian matrix at $\mathbf e_v$ are $J_{11}$ and
$J_{22}$. Recall that we assumed $\mathcal{R}_0>1$. Hence,
$J_{22}<1$. Therefore, equilibrium $\mathbf e_v$ is LAS whenever
$$\mathcal{R}_2=\left(1-w(0)\psi(0)\right)e^\tau<1.$$
\end{itemize}
Hence, the following lemma  holds true.
\begin{lemma}\label{homogeneous-stability}The space homogeneous equilibria
of system (\ref{cooperation})-(\ref{mise-a-jour}) have the
following stability properties.
\begin{enumerate}
    \item The equilibrium $\mathbf e_0=(0,1-\bar{v})$ is unstable.
    \item The  equilibrium $\mathbf e_u=(1,1-\bar{v})$ is LAS whenever $\mathcal{R}_1<1$.
    \item The
     equilibrium
    $\mathbf e_v=(0,0)$ is LAS whenever $\mathcal{R}_2<1$.
\end{enumerate}
\end{lemma}

\subsubsection{Application of the results of \cite{Li2005}}\label{ss}
In the sequel, we study the recursion operator $Q$ defined in
equation (\ref{Q-tau-operator}) and
we check if it satisfies \textbf{Hypotheses 2.1} of Li et
al.  \cite{Li2005}. Recall that $\mathbf P_{n,0}=(u_{n,0},v_{n,0})$
and $Q_{\tau}$ is the time-$\tau$-map solution operator  of
reaction-diffusion system (\ref{cooperation}).
 We consider the order interval $\mathcal{C}_{\mathbf e_{u}}=[\mathbf e_v, \mathbf e_{u}]$, where $\mathbf e_v=\textbf{0}$
and $\mathbf e_{u}=(1,1-\bar{v})$ is the positive coexistence equilibrium
defined in Lemma \ref{homogeneous-stability}.

\begin{lemma}(Some properties of $Q_\tau$)\label{Q-tau-properties}
\begin{enumerate}
    \item The operator $Q_\tau$ is order preserving in the sense that
    if $\textbf{u}$ and $\textbf{v}$ are any two functions in $\mathcal{C}_{\mathbf e_{u}}$ with
    $\textbf{v}\geq \textbf{u}$, then $Q_\tau[\textbf{v}]\geq Q_\tau[\textbf{u}]$.
    \item $Q_\tau$ is translation invariant.
    \item For any $\textbf{v}, \textbf{u}\in \mathcal{C}_{\mathbf e_u}$ and fixed $x$, $|Q_\tau[\textbf{v}](x)-Q_\tau[\textbf{u}](x)|$
    is arbitrarily small, provided
    $|\textbf{v}(y)-\textbf{u}(y)|$ is sufficiently small on a
    sufficiently long interval centered at $x$.
    \item Every sequence $\textbf{v}_n(x)$ in
    $\mathcal{C}_{\mathbf e_{u}}$ has a subsequence
    $\textbf{v}_{n_l}$ such that $Q_\tau[\textbf{v}_{n_l}]$ converges
    uniformly on every bounded set.
\end{enumerate}
\end{lemma}

\begin{proof}
\begin{enumerate}
    \item The reaction-diffusion system (\ref{cooperation}) is a cooperative
system. Hence,  following the analysis at the beginning of this section,  we deduce that the time-$\tau$-map solution
operator  of system (\ref{cooperation}) is order preserving.
    \item Let $\mathbf u$ be the solution of system (\ref{cooperation}) initiated at
$\mathbf u_0$. % and given in Proposition \ref{Proposition-Global-existence}.
 For $y\in \R$, we set
$\mathbf v=T_y[\mathbf u]$. In particular $\mathbf v_0=T_y[\mathbf u_0]$. We have
$\mathbf v_t=(T_y[\mathbf u])_t=T_y[\mathbf u_t]$, $A\mathbf v=AT_y[\mathbf u]=T_y[A\mathbf u]$ and
$\mathbf F(\mathbf v)=\mathbf F(T_y[\mathbf u])=T_y[F(\mathbf u)]$ since $\mathbf F$ does not explicitly depend on
$x\in \R$. Therefore, $\mathbf v_t+A\mathbf v-\mathbf F(\mathbf v)=T_y[\mathbf u_t+A\mathbf u-\mathbf F(\mathbf u)]=0$ and, by the uniqueness of solutions, we have
$$
 T_y[Q_\tau[\mathbf u_0]](x)  =  T_y[\mathbf u](x)  =  \mathbf v(x)  =  Q_\tau[T_y[\mathbf u_0]](x).
$$
Hence the time-$\tau$-map solution operator  of system
(\ref{cooperation}), $Q_\tau$, is translation invariant.

To prove 3. and 4. we write the solution, see e.g. \cite[page 95]{Britton86}, as
 \begin{equation}\label{solution-convolution}
 \left\{
\begin{array}{l}
u=G_{d_u}\star u_0+G_{d_u}\star\star f_{d_u}(u,q),\\
q=G_{d_v}\star q_0+G_{d_v}\star\star f_{d_v}(u,q),\\
\end{array}
\right.
\end{equation}
where
$$
 f_{d_u}(u,q)=u(1-u), \qquad   f_{d_v}(u,q)=-\lambda (q+\bar{v})(1-q-\bar{v})+\lambda\gamma u(1-q-\bar{v}),
$$
$G_d$, $d=d_u,d_v$, as well as the convolution $\star$, were defined in (\ref{Gauss-semigroup})
 and the spatio-temporal convolution is given by
$$
G_d\star \star f_d =\displaystyle\int_{0}^t\frac{1}{\sqrt{4\pi d
(t-s)}}\int_{\R}\exp\left(-\displaystyle\frac{(x-y)^2}{4d
(t-s)}\right)f_d(u(y,s),q(y,s))dyds.
$$
    \item Let $(u,q)$ and $(v,p)$ be two solutions of system (\ref{cooperation}) initiated at
$(u_0,q_0)$ and $(v_0,p_0)$ respectively. We assume that $(u_0,q_0)$, $(v_0,p_0)\in
\mathcal{C}_{\mathbf e_{u}}$, hence $(u,q)$ and $(v,p)$ are also uniformly bounded. Let $\{S_d^c(t)\}_{t\geq 0}$ denotes the (positive) semigroup solving (\ref{C50}) for some function $c$ satisfying $|c(x,t)|\leq L$. Using $|S_d(t)u_0|\leq S_d(t)|u_0|, u_0\in C_{ub}(\R),$ where $\{S_d(t)\}_{t\geq 0}$ is the diffusion semigroup (\ref{Gauss-semigroup}), and the Phillips-Dyson expansion to (\ref{C50}) we ascertain that for any $u_0$
$$
|S^c_d(t)u_0| \leq S^L_d(t)|u_0| = e^{Lt}S_d(t)|u_0|.
$$
As before, we begin with solutions $u$ and $v$ to (\ref{cooperation})$_1$. Repeating the argument leading to (\ref{C8}), we see that $z(x,t) = u(x,t)-v(x,t)$ can be estimated as
 $$|z(x,t)| \leq e^{Lt}[S_{d_u}(t)|u_0-v_0|](x) + \displaystyle\frac{e^{Lt}}{\sqrt{4\pi d_u
t}}\displaystyle\int_{\R}\exp\left(-\displaystyle\frac{(x-y)^2}{{4d_u
t}}\right)|u_0(y)-v_0(y)| dy.
$$
 Let, for $\epsilon>0$, $r>0$ be such that
 $$
 \left(\int_{-\infty}^\frac{-r}{2\sqrt{d_u\tau}} + \int^{\infty}_\frac{r}{2\sqrt{d_u\tau}}\right) e^{-z^2}dz \leq \frac{\epsilon \sqrt \pi}{4e^{L\tau}}.
 $$
 Then let us fix $x$ and let $|u_0(x)-v_0(x)| \leq  \delta \leq \epsilon/2e^{L\tau} $ on $(x-r,x+r)$ so that we obtain for $0<t\leq \tau$
 \begin{eqnarray*}
 |z(x,t)| &\leq & \displaystyle\frac{e^{Lt}}{\sqrt{4\pi d_u
t}}\displaystyle\int_{x-r}^{x+r}\exp\left(-\displaystyle\frac{(x-y)^2}{{4d_u
t}}\right)|u_0(y)-v_0(y)| dy \\
&&+ \displaystyle\frac{e^{Lt}}{\sqrt{4\pi d_u
t}}\displaystyle\left(\int_{-\infty}^{x-r}+\int_{x+r}^\infty\right)\exp\left(-\displaystyle\frac{(x-y)^2}{{4d_u
t}}\right)|u_0(y)-v_0(y)| dy \\
&\leq& e^{Lt} \delta +  \displaystyle\frac{2 e^{Lt}}{\sqrt\pi }\displaystyle\left(\int_{-\infty}^\frac{-r}{2\sqrt{d_ut}}+\int_\frac{r}{2\sqrt{d_ut}}^\infty\right)e^{-z^2} dz\leq e^{L\tau} \delta+ \displaystyle\frac{2 e^{L\tau}}{\sqrt\pi }\displaystyle\left(\int_{-\infty}^\frac{-r}{2\sqrt{d_u\tau}}+\int_\frac{r}{2\sqrt{d_u\tau}}^\infty\right)e^{-z^2} dz\leq \epsilon.\label{jbest1}
 \end{eqnarray*}
By choosing appropriate $r$ we see that the estimate is valid for $x$ in any given bounded subset of $\R$.

Considering now (\ref{cooperation})$_2$, we see that $Z(x,t) = q(x,t)-p(x,t)$ is a solution to
 \begin{eqnarray}
 Z_t &=& d_v Z_{xx} + (q-qu-q^2-p+pv + p^2) = d_v Z_{xx} + Z(1-(p+q) - u) - p(u-v),\nonumber\\
 Z(x,0) &=& q(x,0)-p(x,0) =: Z_0(x,0)
 \label{Z1}
 \end{eqnarray}
 and considerations as above show that $|Z(x,t)| \leq e^{L_1t}\Psi(x,t)$, where
  \begin{eqnarray}
 \Psi_t &=& d_v \Psi_{xx}+ e^{-L_1 t}z,\nonumber\\
 \Psi(x,0) &=& |Z_0(x,0)|.
 \label{Z2}
 \end{eqnarray}
 In the above, $L_1$ is a constant bounding $|1-(p+q)-u|, 0\leq p,q,u \leq 1$ and we used $0\leq v\leq 1$. Hence
 \begin{eqnarray*}
 |Z(x,\tau)| &\leq & \displaystyle\frac{e^{L_1\tau}}{\sqrt{4\pi d_v
\tau}}\displaystyle\int_{\R}\exp\left(-\displaystyle\frac{(x-y)^2}{{4d_v
\tau}}\right)|q_0(y)-p_0(y)| dy\\
&&+ \displaystyle\int_{0}^\tau \frac{e^{L_1(\tau-s)}}{\sqrt{4\pi d_v
(\tau-s)}}\int_{\R}\exp\left(-\displaystyle\frac{(x-y)^2}{4d_v
(\tau-s)}\right)z(y,s)dyds
\end{eqnarray*}
and the estimates follow as above where, in the second term, we use the fact that (\ref{jbest1}) is uniform on $[0,\tau]$ and any bounded subset of $\R$.
    \item For each $t\in]0, \tau]$, the functions $Q_t[\mathbf w_0]$ with
$\mathbf w_0=(u_0,q_0)\in\mathcal{C}_{\mathbf e_{u}}$ form an equicontinuous family.
Indeed, for $0<t\leq \tau$, $\mathbf w_0\in\mathcal{C}_{\mathbf e_{u}}$ and $x\in\R$, $Q_{t}[\mathbf w_0(x)]=:\mathbf w(t,x)=(u(t,x),q(t,x))$, following (\ref{solution-convolution}) and by using the
property of the spatial convolution, we obtain
 \begin{equation}
 \left\{
\begin{array}{l}
\displaystyle\frac{\partial u}{\partial x}=\displaystyle\frac{\partial G_{d_u}}{\partial x}\star u_0+\displaystyle\frac{\partial G_{d_u}}{\partial x}\star\star  f_{d_u}(u,q),\\
\displaystyle\frac{\partial q}{\partial x}=\displaystyle\frac{\partial G_{d_v}}{\partial x}\star q_0+\displaystyle\frac{\partial G_{d_v}}{\partial x}\star\star f_{d_v}(u,q).\\
\end{array}
\right.
\end{equation}
Since $(u_0, q_0)\in \mathcal{C}_{\mathbf e_{u}}$ i.e. $0\leq u_0\leq
1$ and $0\leq q_0\leq 1-\bar{v}$, we have $\|f_{d_u}(u,q)\|_{\infty}\leq 1$ and
$\|f_{d_v}(u,q)\|_{\infty}\leq M$ for some $M$. In addition, by direct calculation or, more generally, by
\cite[Theorem 11]{Friedman1964}, for $d=d_u,d_v$, there exist
positive constants $\alpha_d$ and $\beta_d$ such that for $t>0$
$$\left|\displaystyle\frac{\partial G_d(x,t)}{\partial x}\right|\leq
\frac{\alpha_de^{-\beta_d\frac{x^2}{t}}}{t}$$
 Hence,
\begin{equation}\label{derivation-solution-convolution1}
 \left\{
\begin{array}{l}
\displaystyle\left|\frac{\partial u}{\partial x}(x,t)\right|\leq \displaystyle\int_{\R}\frac{\alpha_{d_u}e^{-\beta_{d_u}\frac{(x-y)^2}{t}}}{t}dy+\int_0^t\int_{\R}\frac{\alpha_{d_u}e^{-\beta_{d_u}\frac{(x-y)^2}{t-s}}}{t-s}dyds,\\
\displaystyle\left|\frac{\partial q}{\partial x}(x,t)\right|\leq\displaystyle(1-\bar{v})\int_{\R}\frac{\alpha_{d_v}e^{-\beta_{d_v}\frac{(x-y)^2}{t}}}{t}dy+M\int_0^t\int_{\R}\frac{\alpha_{d_v}e^{-\beta_{d_v}\frac{(x-y)^2}{t-s}}}{t-s}dyds.
\end{array}
\right.
\end{equation}
Evaluating the integrals in (\ref{derivation-solution-convolution1}) we obtain
\begin{equation}\label{derivation-solution-convolution3}
 \left\{
\begin{array}{l}
\left|\displaystyle\frac{\partial u}{\partial x}(x,t)\right|\leq \alpha_{d_u}\sqrt{\displaystyle\frac{\pi}{\beta_{d_u}t}}+\alpha_{d_u}\sqrt{\displaystyle\frac{\pi t}{\beta_{d_u}}}=:\delta_{d_u}(t),\\
\left|\displaystyle\frac{\partial q}{\partial x}(x,t)\right|\leq (1-\bar{v})\alpha_{d_v}\sqrt{\displaystyle\frac{\pi}{\beta_{d_v}t}}+M\alpha_{d_v}\sqrt{\displaystyle\frac{\pi t}{\beta_{d_v}}}=:\delta_{d_v}(t),
\end{array}
\right.
\end{equation}
where  $\delta_{d_u}(t)$ and $\delta_{d_v}(t)$ do not depend on $u_0$,
$q_0$ and $x\in\R$. Thus the first spatial
derivative of the solution $\mathbf w=(u,q)$ is uniformly bounded. Hence, using the mean value theorem, we deduce that the family of solutions of system
(\ref{cooperation}) is equicontinuous. Then, part 4 of Lemma \ref{Q-tau-properties} follows from the
Arzela-Ascoli's theorem, see e.g. % Corollary 41, page 169 of Royden
\cite[Corollary 41]{Royden1988} (i.e. any bounded and equicontinuous sequence of
continuous functions on a separable metric space contains a uniformly
convergent subsequence on every bounded subset).
\end{enumerate}
\end{proof}

In order to obtain properties of the operator $Q$ defined in
(\ref{Q-tau-operator}), we first
formulate results for the nonlinear operator $\mathbf H: B\rightarrow
B$. Recall that the Banach space considered here is
$B=C_{ub}(\R)\times C_{ub}(\R)$ endowed with the sup-norm. For
$\textbf{u}=(u,q)\in B$,
$$\mathbf H(\textbf{u})=((1-w(q)\psi(u))u;(1-\eta)q+\eta(1-\bar{v})).$$

\begin{lemma}(Some properties of $\mathbf H$)\label{H-properties}
\begin{enumerate}
    \item The nonlinear operator $\mathbf H$ is order preserving in the sense that
    if $\textbf{u}$ and $\textbf{v}$ are any two functions with
    $\textbf{v}\geq \textbf{u}$, then $\mathbf H(\textbf{v})\geq \mathbf H(\textbf{u})$.
    \item $\mathbf H$ is translation invariant.
    \item For any two functions $\textbf{u}$ and $\textbf{v}$,
    $\|\mathbf H(\textbf{v})-\mathbf H(\textbf{u})\|_B \leq C \|\textbf{v}-\textbf{u}\|_B$
     where $C\in \R$ and depend on $\|\textbf{u}\|_B$, $\|\textbf{v}\|_B$.
    \item If a sequence $\textbf{v}_n(x)$ converges
    uniformly on every bounded set, then $\mathbf H(\textbf{v}_n(x))$ also has the same property.
\end{enumerate}
\end{lemma}

\begin{proof}
Let $\textbf{u}=(u,q)$ and $\textbf{v}=(v,p)$ be such that
$\textbf{u}\leq\textbf{v}$. Hence
$(1-w(q)\psi(u))u\leq(1-w(p)\psi(v))v$ since $w$ (resp.
$\psi$) is increasing (resp. decreasing) and
$(1-\eta)q+\eta(1-\bar{v})\leq(1-\eta)p+\eta(1-\bar{v})$. Thus
$\mathbf H(\textbf{u})\leq H(\textbf{v})$ and part 1 of Lemma
\ref{H-properties} holds. Since $\mathbf H$ does not explicitly depend on
$x\in \R$, then $\mathbf H$ is translation invariant and part 2 of Lemma
\ref{H-properties} is valid. Part 3 follows from the local Lipschitz
property of $\mathbf H,$ while part 4 follows from the continuity of $\mathbf H$.
This ends the proof.
\end{proof}

Combining Lemmas \ref{Q-tau-properties} and \ref{H-properties}, we
deduce the following result for the recursion operator $Q:=\mathbf H\circ Q_\tau$, defined
in (\ref{Q-tau-operator}).

\begin{lemma}(Some properties of $Q$)\label{Q-properties}
\begin{enumerate}
    \item The operator $Q$ is order preserving in the sense that
    if $\textbf{u}$ and $\textbf{v}$ are any two functions in $\mathcal{C}_{\mathbf e_{u}}$ with
    $\textbf{v}\geq \textbf{u}$, then $Q[\textbf{v}]\geq Q[\textbf{u}]$.
    \item $Q$ is translation invariant.
    \item For any fixed $x$, $|Q[\textbf{v}](x)-Q[\textbf{u}](x)|$
    is arbitrarily small, provided
    $|\textbf{v}(y)-\textbf{u}(y)|$ is sufficiently small on a
    sufficiently long interval centered at $x$.
    \item Every sequence $\textbf{v}_n(x)$ in
    $\mathcal{C}_{\mathbf e_{u}}$ has a subsequence
    $\textbf{v}_{n_l}$ such that $Q[\textbf{v}_{n_l}]$ converges
    uniformly on every bounded set.
\end{enumerate}
\end{lemma}

In the sequel, we assume that $\mathcal{R}_0=(1-\eta)e^{\lambda\tau}>1$ and
$\mathcal{R}_1=(1-\eta)e^{\lambda\tau(1-\gamma)}<1$; that is, the coexistence equilibrium
$\mathbf e_{u}=(1,1-\bar{v})$ exists and is stable. We also assume
that $\mathcal{R}_2=\left(1-w(0)\psi(0)\right)e^\tau>1$; that is,  $\mathbf e_v=(0,0)$ is unstable.

Taking into account Lemma \ref{Q-properties}, we deduce that the
recursion operator $Q$ defined in (\ref{Q-tau-operator}) verifies all conditions of \textbf{Hypothesis 2.1.} Consequently, we can apply the results
of Li et al.  \cite{Li2005} that deal with the spreading speeds
and existence of traveling wave solutions for systems
(\ref{cooperation})-(\ref{mise-a-jour}). Recall that a traveling wave of speed $c$ is a solution of the
recursion (\ref{Q-tau-operator}) which has the form
$\textbf{u}_n(x,0)=\textbf{Z}(x-nc)$ with $\textbf{Z}(s)$ being a function
in $\mathcal{C}_{\mathbf e_{u}}$. That is, the solution at time $n + 1$ is
simply the translate by $c$ of its value at $n$. Using the definition of $c^*$ (see (\ref{c-etoile-general})) and $c^*_f$ (see
(\ref{c-etoile-f-general})), the following result holds true.

\begin{theorem}\label{theorem}(Spreading speeds and traveling waves)
\begin{itemize}
    \item[(a).]\textbf{Slowest spreading speed}: There is an index $j\in\{1,2\}$ for which the following statement is
true: Suppose that the initial function $\textbf{u}_0(x)$ is
\textbf{\emph{0}} for all sufficiently large $x$, and that there are
positive constants $0<\rho\leq\sigma<1$ such that
$\textbf{\emph{0}}\leq\textbf{u}_0\leq\sigma\mathbf  e_{u}$ for all $x$ and
$\textbf{u}_0\geq\rho \mathbf  e_{u}$ for all sufficiently negative $x$. Then
for any positive $\varepsilon$ the solution $\textbf{u}_n$ of the
recursion (\ref{Q-tau-operator}) has the properties
\begin{equation}\label{slowest-speed}
    \lim\limits_{n\rightarrow+\infty}\left[\sup\limits_{x\geq
    n(c^*+\varepsilon)}\{\textbf{u}_n\}_j(x)\right]=0
\end{equation}
and
\begin{equation}\label{slowest-speed-2}
    \lim\limits_{n\rightarrow+\infty}\left[\sup\limits_{x\leq
    n(c^*-\varepsilon)}\{\mathbf e_{u}-\textbf{u}_n(x)\}\right]=\textbf{\emph{0}};
\end{equation}
that is, the $j$th component spreads at a speed no higher than
$c^*$, and no component spreads at a lower speed.
\item[(b).]\textbf{Fastest spreading speed}: There is an index $i\in\{1,2\}$ for which the following statement is
true: Suppose that the initial function $\textbf{u}_0(x)$ is
\textbf{\emph{0}} for all sufficiently large $x$, and that there are
positive constants $0<\rho\leq\sigma<1$ such that
$\textbf{\emph{0}}\leq\textbf{u}_0\leq\sigma \mathbf e_{u}$ for all $x$ and
$\textbf{u}_0\geq\rho \mathbf e_{u}$ for all sufficiently negative $x$. Then
for any positive $\varepsilon$ the solution $\textbf{u}_n$ of the
recursion (\ref{Q-tau-operator}) has the properties
\begin{equation}\label{fastest-speed}
    \limsup\limits_{n\rightarrow+\infty}\left[\inf\limits_{x\leq
    n(c^*_f-\varepsilon)}\{\textbf{u}_n\}_i(x)\right]>0
\end{equation}
and
\begin{equation}\label{fastest-speed-2}
    \lim\limits_{n\rightarrow+\infty}\left[\sup\limits_{x\geq
    n(c^*_f+\varepsilon)}\textbf{u}_n(x)\right]=\textbf{\emph{0}};
\end{equation}
that is, the $i$th component spreads at a speed no less than
$c^*_f$, and no component spreads at a higher speed.
\item[(c).]\textbf{Monostable traveling wave}: If $c\geq c^*$, there is a non-increasing monostable traveling wave solution
$\textbf{Z}(x-nc)$ of speed $c$ with $\textbf{Z}(-\infty)=\mathbf e_{u}$ and
$\textbf{Z}(+\infty)$ an equilibrium other than $\mathbf e_{u}$.
\par
If there is a traveling wave $\textbf{Z}(x-nc)$ with
$\textbf{Z}(-\infty)=e_{u}$ such that for at least one component
$i\in\{1,2\}$ $$\liminf\limits_{x\rightarrow\infty}Z_i(x)=0,$$ then
$c\geq c^*$. If this property is valid for all components of
$\textbf{Z}$, then $c\geq c^*_f$.
\end{itemize}
\end{theorem}

\begin{proof}
Since  the
recursion operator $Q$ defined in (\ref{Q-tau-operator}) verifies \textbf{Hypothesis 2.1.}, the proof of Theorem \ref{theorem} follows directly from Theorems 2.1, 2.2 and 3.1 of Li et al. \cite{Li2005}.
\end{proof}

Let us point out that if, instead of the first coordinates
change (\ref{new-variables1}), we considered
\begin{equation}\label{new-variables-new1}
\left\{
    \begin{array}{ccl}
      u_n & = & 1-U_n, \\
      v_n & = & V_n,
    \end{array}
\right.
\end{equation}
then we would obtain a monotone increasing system (see Appendix
\ref{appendix-scaling-grassland}). Hence, by reasoning as before,
one can study the case where the equilibrium $\textbf{0}$
is stable and the equilibrium $\mathbf e_{u}$ is unstable. However,
the bistable case, i.e. when both \textbf{0} and $\mathbf e_{u}$ are
simultaneously stable, remains an open problem.

\subsection{Numerical simulations}
In this section we provide numerical simulations of the
 impulsive tree-grass reaction-diffusion model (\ref{I-savnna})-(\ref{pulsed_swv_eq2}). We note that the parameters with $\tilde{}$ below refer to this system and can be derived from the corresponding parameters related to the normalized system
(\ref{cooperation})-(\ref{mise-a-jour}), see Appendix \ref{Appendix-scaling}. Thus, we consider fire events as periodic and pulse perturbations with the    time period $\tilde{\tau}$.
 The form of the functions $\gamma_{G}(\textbf{W})$, $\gamma_{T}(\textbf{W})$,
 $K_{G}(\textbf{W})$, $K_{T}(\textbf{W})$, $\psi$ and  $w_G$ is considered following Yatat et al.  \cite{Yatat2018}. The readers are  referred to Appendix \ref{Appendix-scaling} for their definition and parametrization. The parameter values used in the following simulations are also given in Appendix \ref{Appendix-scaling}: see Tables \ref{table-param-1} and \ref{table-param-2}.

%
% Spreading de E_G=(0; G_bar) (instable) vers E_T=(T_bar;0) (stable)
%
Using the parameters values given in Table \ref{table-param-1}, page \pageref{table-param-1}, Fig. \ref{spreadingl} depicts the spreading of tree and grass
biomasses toward the stable forest homogeneous steady state
$\mathbf E_T=(273.3955,0)$. %A transient state is first observed before the stationary state.
In this case, $\tilde{\mathcal{R}}_0=1.3667$,
$\tilde{\mathcal{R}}_1=0.0058$ and $\tilde{\mathcal{R}}_2=3.3801$.
Recall that $\mathbf E_T$ is LAS whenever $\tilde{\mathcal{R}}_1<1$, while
the grassland homogeneous steady state $\mathbf E_G=(0,3.3888)$ exists when
$\tilde{\mathcal{R}}_0>1$ and is LAS whenever
$\tilde{\mathcal{R}}_2<1$. In terms of tree-grass interactions,
Fig. \ref{spreadingl} illustrates the spreading of forest or the
so-called 'forest encroachment' phenomenon (Yatat et al.
\cite{Yatat2017b}).
\begin{figure}[H]
    \centering
\subfloat[][]{
\includegraphics[scale=0.52]{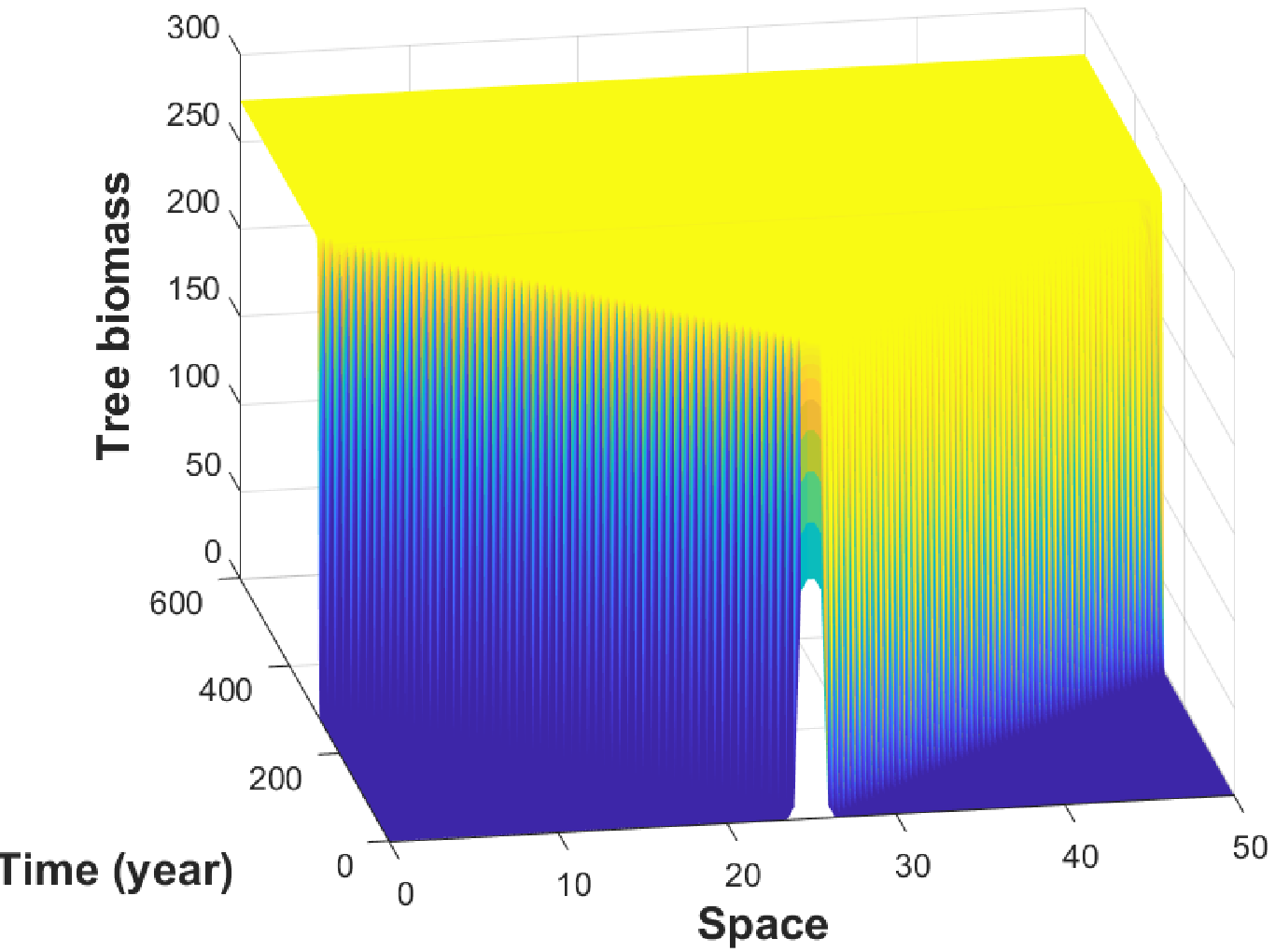}}
\vspace{0.5cm} \subfloat[][]{
\includegraphics[scale=0.52]{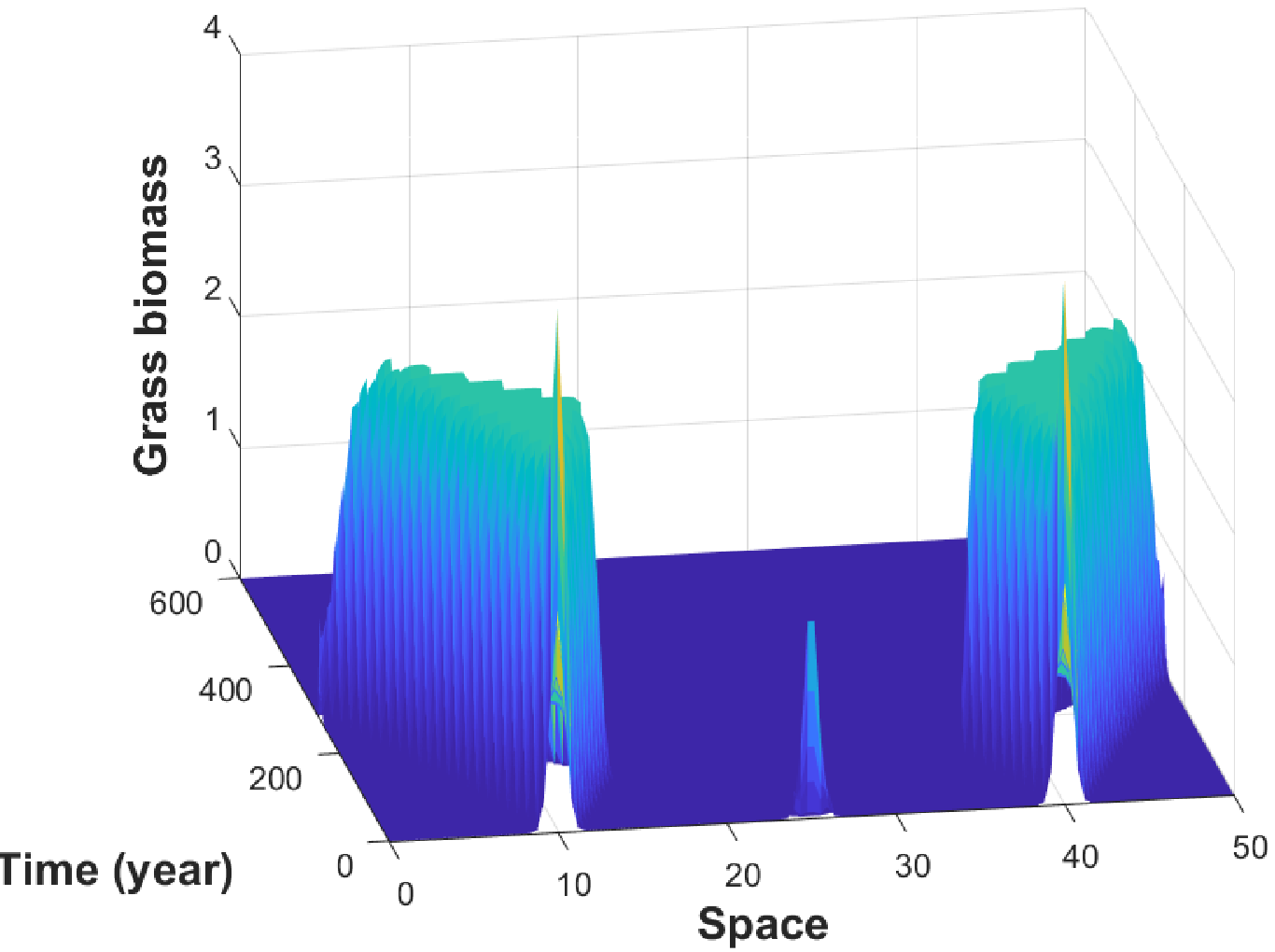}}
\caption{ Illustration of the spreading of both tree (see panel (a))
and grass (see panel (b)) biomasses toward the stable forest
equilibrium $\mathbf E_T=(273.3955,0)$ with system
(\ref{I-savnna})-(\ref{pulsed_swv_eq2}). $\textbf{W}=1200$ mm.yr$^{-1}$,
$\tilde{\tau}=2$ yr, $d_T$=0.001 and
$d_G$=0.002. Remaining parameters are in Table
\ref{table-param-1}, page \pageref{table-param-1}.}\label{spreadingl}
 \end{figure}
In the setting of the forest encroachment phenomenon, we carry out numerical simulations to compute the spreading speed of forest
biomass. We investigate the relationship between the tree biomass
diffusion coefficient and its spreading speed. To estimate the
spreading speed of the tree biomass that undergoes a forest
encroachment, grass biomass diffusion coefficient $d_G$ is kept
constant and equal to 0.002 while tree biomass diffusion coefficient
$d_T$ varies in the range $[0.001, 0.9]$. In the diffusive logistic
equation, a linear relationship is obtained between the wave speed
and the square root of the diffusion coefficient (e.g. Volpert \cite{Volpert2014}, Yatat et al.
\cite{Yatat2017b}, Yatat and Dumont \cite{YatatDumont2018}). Hence,
for the tree biomass, we consider an equation of the form
$c_T(d_T)=a_1d_T^{a_2}$ to be fitted for the data shown in Fig. \ref{speed_fitting1}(b), page \pageref{speed_fitting1}, where $c_T\in[0.0865, 1.6899]$. We found that
$a_1\in(1.7624, 1.7861)$ and $a_2\in(0.4816, 0.4911)$ with 95\%
confidence. In fact, $a_1=1.7743$ and $a_2=0.48634$, with $r^2=1$, indicating that 100\% of the variance of the data is
explained by the equation.
\begin{figure}[H]
    \centering
        \subfloat[][]{  \includegraphics[scale=0.52]{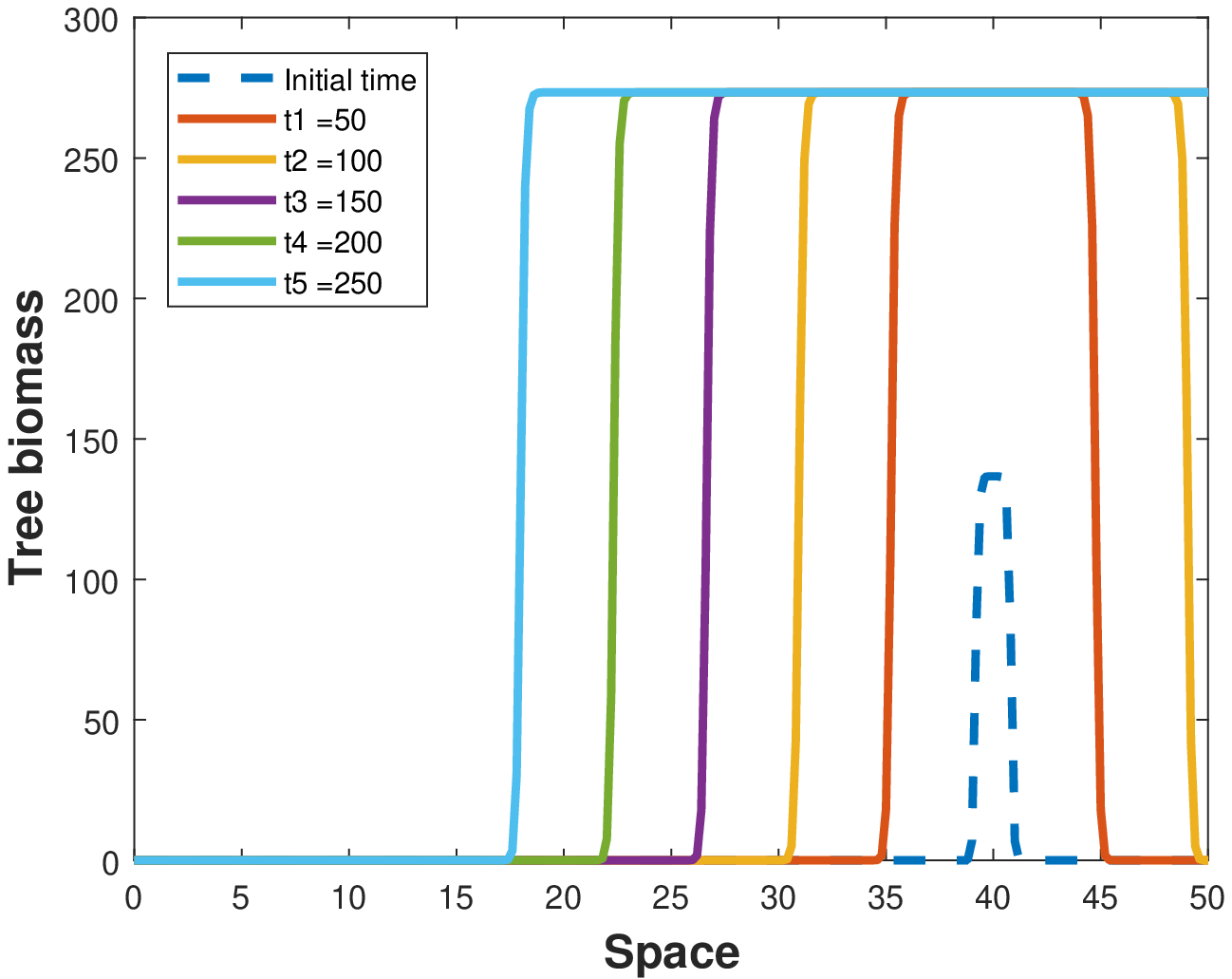}}
        \vspace{0.5cm}
        \subfloat[][]{  \includegraphics[scale=0.52]{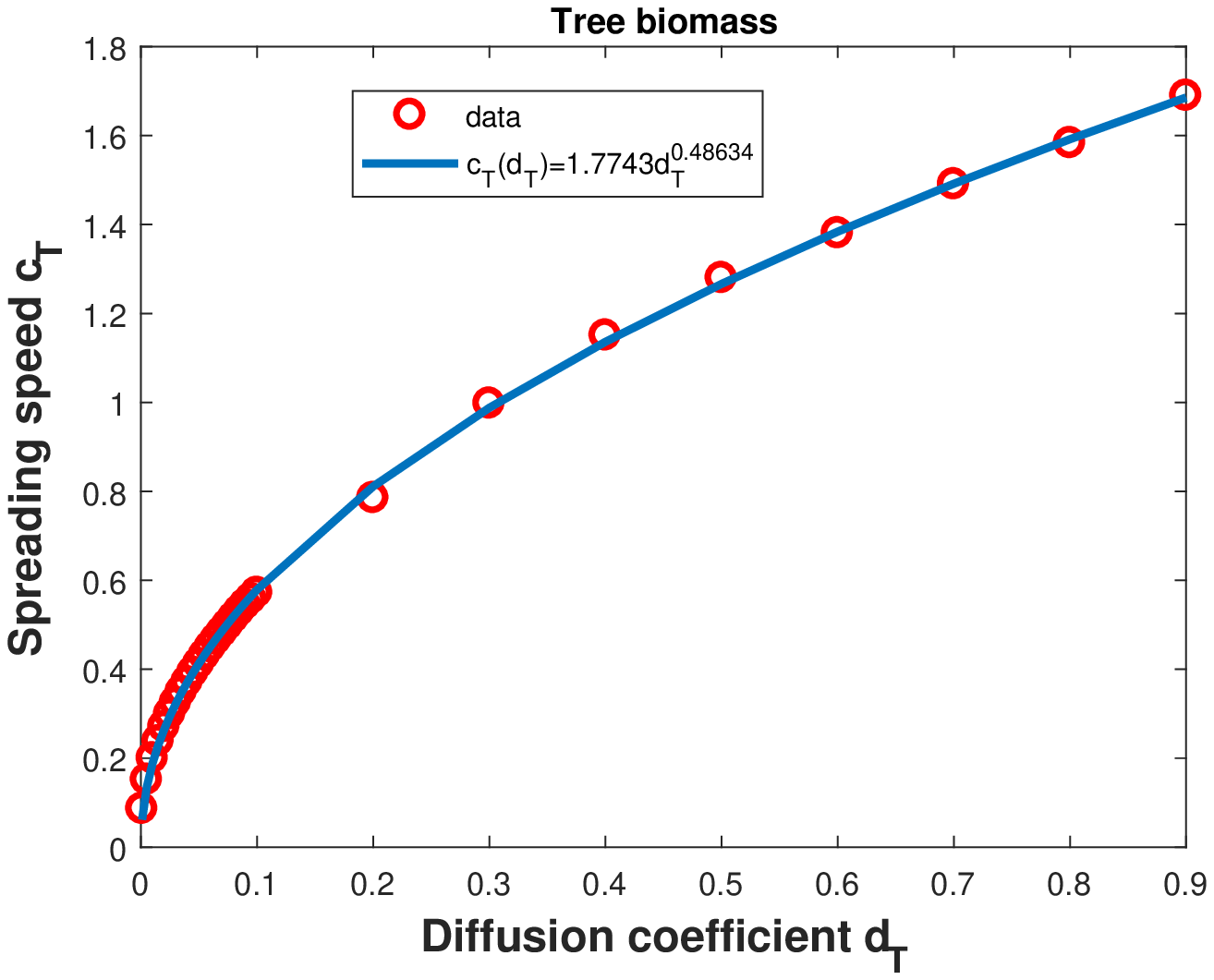}}
        \caption{The forest homogeneous steady state $\mathbf E_T=(273.3955,0)$ is stable while the
        grassland homogeneous steady state $\mathbf E_G=(0,3.3888)$ is
        unstable. We illustrate the forest encroachment phenomenon  (see panel (a)) with $\textbf{W}=1200$ mm.yr$^{-1}$,
$\tilde{\tau}=2$ yr, $d_{T}$=0.001,
$d_G$=0.002 and, we carry out spreading speeds fitting for
        tree biomass (see panel (b)).
         Remaining parameters are in Table \ref{table-param-1}, page \pageref{table-param-1}.}\label{speed_fitting1}
 \end{figure}
With the parameter values given in Table \ref{table-param-2}, page \pageref{table-param-2}, Fig. \ref{spreading2} illustrates the spreading of both tree
and grass biomasses toward the grassland homogeneous steady state
$\mathbf E_G=(0, 2.1096)$. %A transient state is first observed before the stationary state.
In this case, $\tilde{\mathcal{R}}_0=2.0425$,
$\tilde{\mathcal{R}}_1=1.2541$ and $\tilde{\mathcal{R}}_2=0.9932$.
Recall that $\mathbf E_G$ exists when $\tilde{\mathcal{R}}_0>1$ and is LAS
whenever $\tilde{\mathcal{R}}_2<1$. %In terms of tree-grass
%interactions, Figure \ref{spreading2}  illustrates the
%spreading of grassland.
 We further investigate the relationship
between the diffusion coefficient and the spreading speed of the
grass biomass. We assume that $d_T=0.001$ and
 $d_G\in[0.001, 1]$.
 Motivated by the linear relationship obtained between the wave speed
and the square root of the diffusion coefficient in the diffusive
logistic equation, an equation like
$c_G(d_G)=a_1d_G^{a_2}$ was fitted to the data shown in Fig.
\ref{spreading2bis}(b), page \pageref{spreading2bis}, where $c_G\in[0.0465, 1.0877]$. We found that
$a_1\in(1.0829, 1.0935)$ and $a_2\in(0.4839, 0.4915)$ with 95\%
confidence. In fact $a_1=1.0882$ and $a_2=0.48769$, with $r^2=1$,  indicating that 100\% of the variance of the data is
explained by the equation.

\begin{figure}[H]
    \centering
        \subfloat[][]{  \includegraphics[scale=0.52]{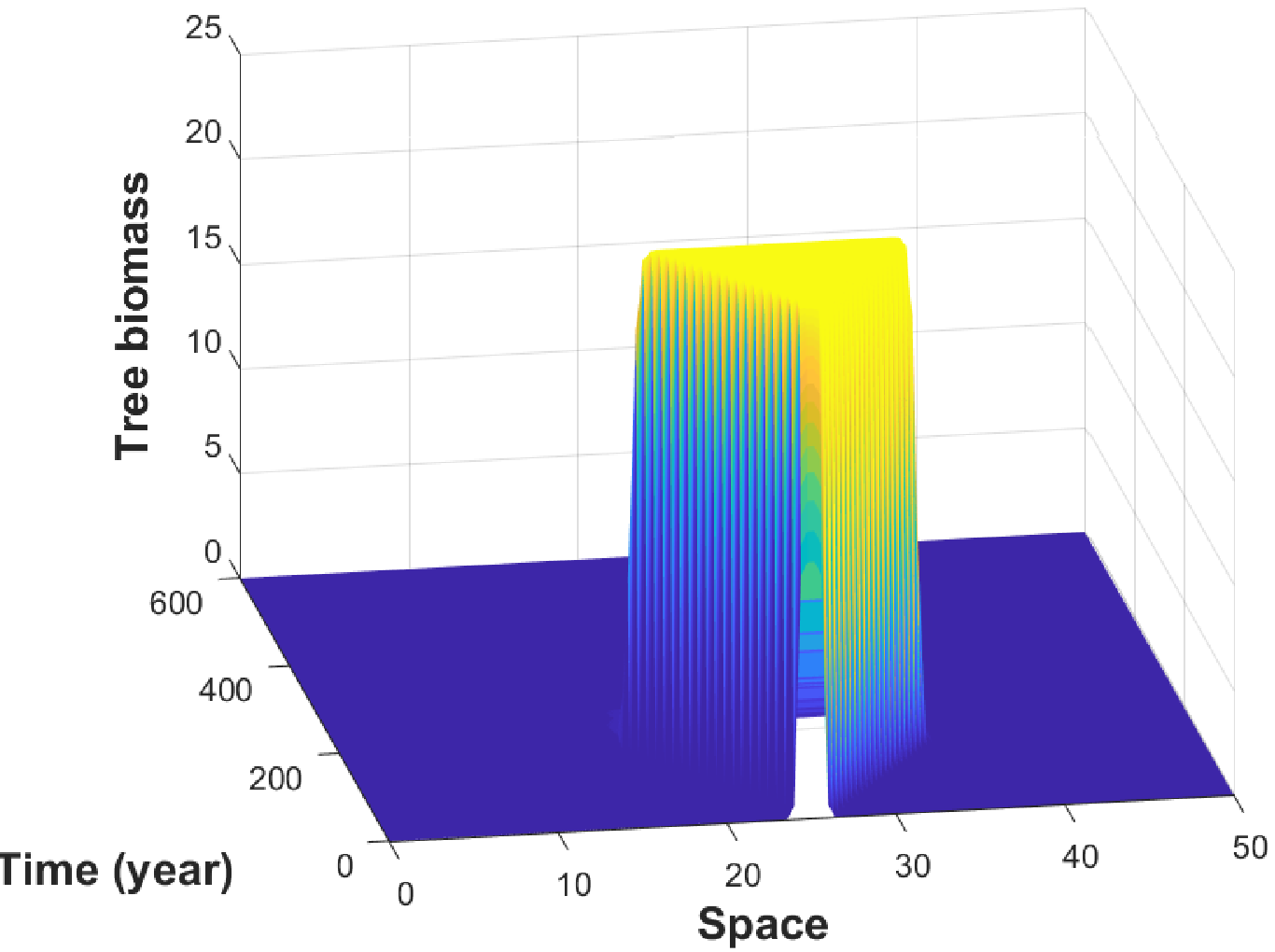}}
        \vspace{0.5cm}
        \subfloat[][]{  \includegraphics[scale=0.52]{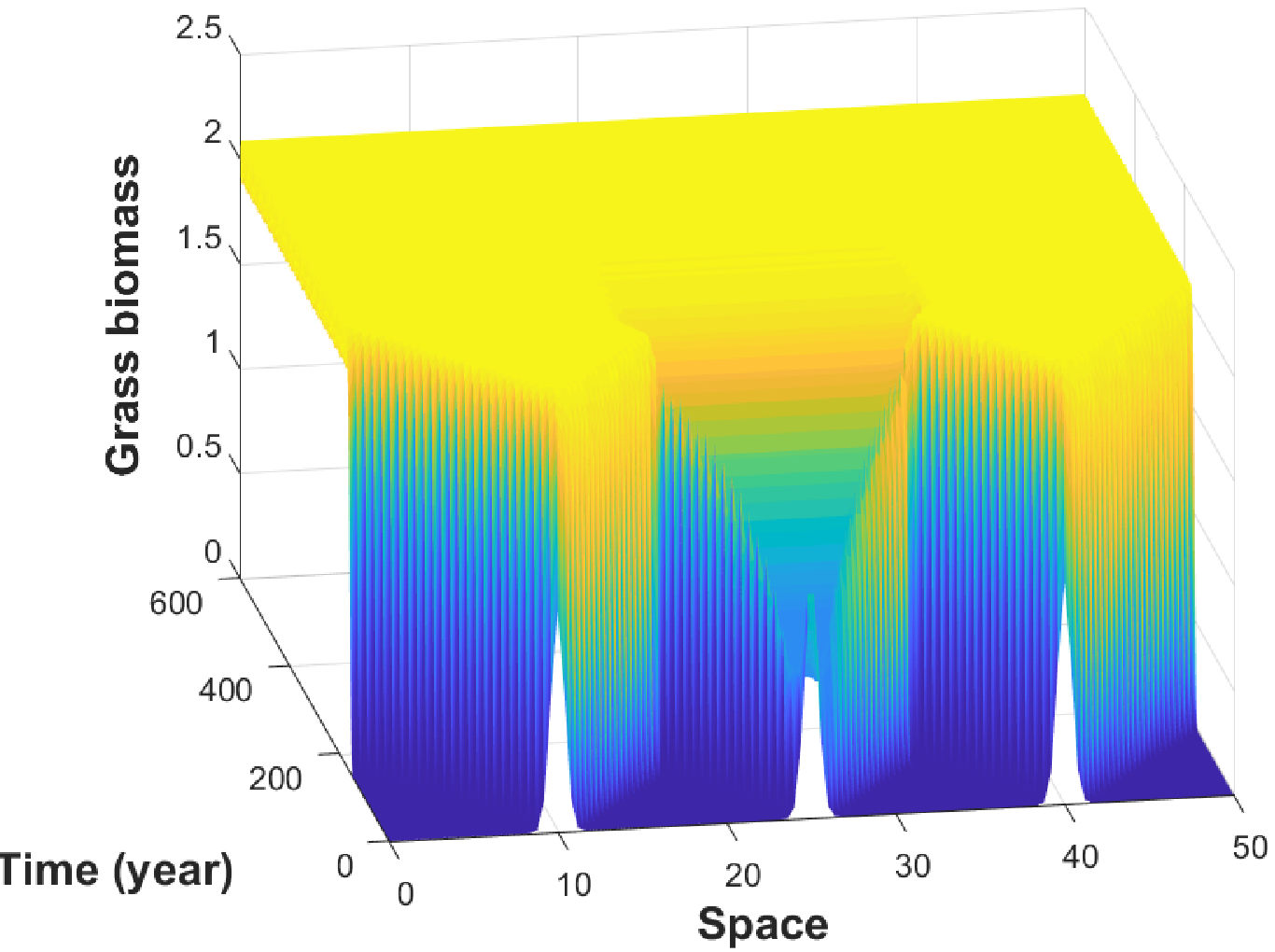}}
        \caption{ Illustration of the spreading of both tree (panel (a)) and
        grass (panel (b)) biomasses
        toward the grassland equilibrium $\mathbf E_G=(0, 2.1096)$. $\textbf{W}=450$ mm.yr$^{-1}$,
$\tilde{\tau}=2$ yr, $d_{T}$=0.001,
$d_G$=0.002. Remaining parameters are in Table
\ref{table-param-2}, page \pageref{table-param-2}.}\label{spreading2}
 \end{figure}

 \begin{figure}[H]
    \centering
        \subfloat[][]{  \includegraphics[scale=0.52]{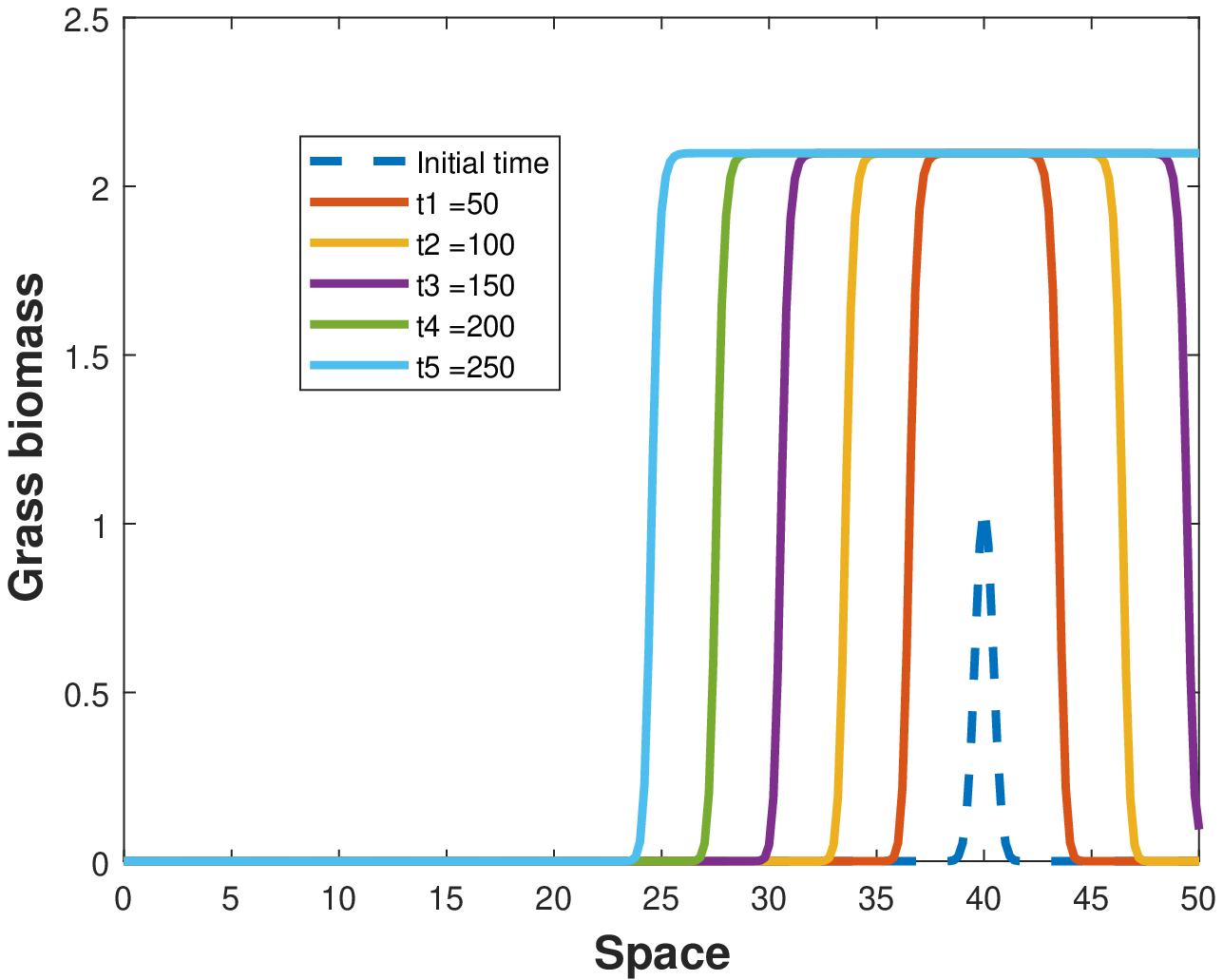}}
        \vspace{0.5cm}
        \subfloat[][]{  \includegraphics[scale=0.52]{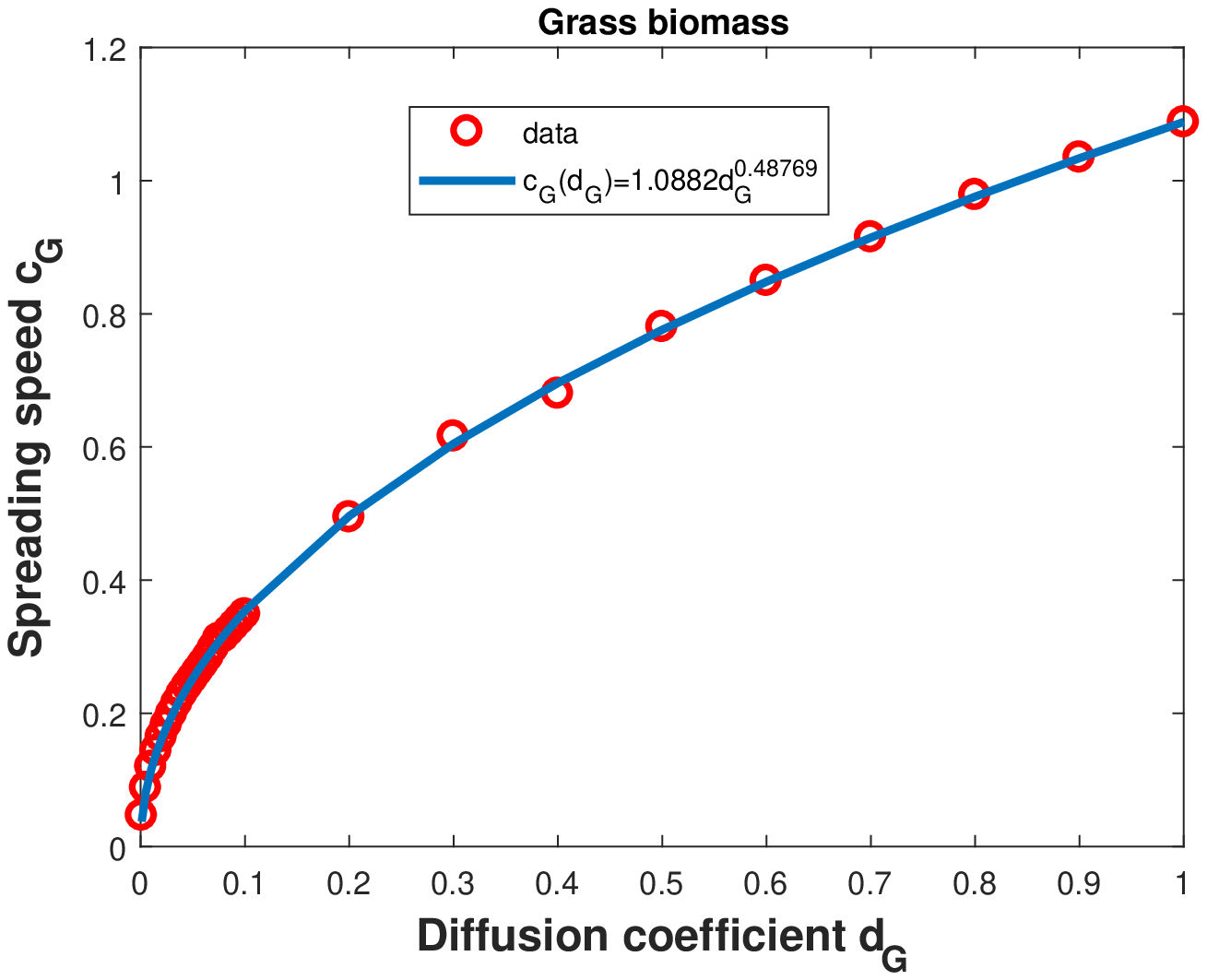}}
         \caption{The forest homogeneous steady state $\mathbf E_T=(24.3905,0)$ is unstable while the
        grassland homogeneous steady state $\mathbf E_G=(0, 2.1096)$ is
        stable. We illustrate the grassland encroachment phenomenon (see panel (a)) and spreading speeds fitting for
        grass biomass (see panel (b)). In panel (a), $\textbf{W}=450$ mm.yr$^{-1}$,
$\tilde{\tau}=2$ yr, $d_{T}$=0.001,
$d_G$=0.002. Remaining parameters are in Table \ref{table-param-2}, page \pageref{table-param-2}.}\label{spreading2bis}
 \end{figure}

\section{Conclusion}\label{Conclusion}
In this paper, we used the vector-valued recursion equation theory
(e.g. Weinberger et al.  \cite{Weinberger2002}, Lewis et al.
 \cite{Lewis2002}, Li et al.  \cite{Li2005}) to propose
a framework that deals with the existence of traveling waves for
monotone systems of impulsive reaction-diffusion equations and with the computation of
spreading speeds. This
study extends a previous one that dealt with the impulsive Fisher-Kolmogorov-Petrowsky-Piscounov (FKPP) equation (Yatat and
Dumont \cite{YatatDumont2018}). Specifically, our results handle the aforementioned issues only in the case of monostable situations; that is, when only one of the equilibria is stable. However, the bistable case (i.e. when two equilibria are simultaneously stable) is also meaningful and needs to be studied for monotone systems of impulsive reaction-diffusion equations. Travelling waves in bistable reaction-diffusion systems without impulsive perturbations are treated in Volpert  \cite{Volpert2014} (see also Yatat et al.  \cite{Yatat2017} for application in the context of bistable tree-grass reaction-diffusion model).

\par

The computation of
spreading speeds and the existence of traveling waves for
bistable monotone systems of impulsive reaction-diffusion equations will be the aim of future studies. It first requires to  elaborate a recursion equations theory that includes bistable cases and thus extending the results of Li et al.  \cite{Li2005} that only deal with monostable cases.

\bibliographystyle{plain}
\bibliography{bibliography}

\appendix

\section{Normalization procedure and model's parameter values}\label{Appendix-scaling}
We first assume that there are no fire events. In this setting,
equations (\ref{I-savnna})-(\ref{pulsed_swv_eq2}) become

\begin{equation}\label{savanna-RD}
\left\{
\begin{array}{l}
\left.
\begin{array}{l}
\displaystyle\frac{\partial T}{\partial t}=d_T(\textbf{W})\displaystyle \frac{\partial^2 T}{\partial x^2}+\gamma_{T}(\textbf{W})\left(1-\displaystyle\frac{T}{K_{T}(\textbf{W})}\right)T-\delta_{T}T,\\
\\
\displaystyle \frac{\partial G}{\partial t}=d_G(\textbf{W})\displaystyle \frac{\partial^2 G}{\partial x^2}+\gamma_{G}(\textbf{W})\left(1-\displaystyle\frac{G}{K_{G}(\textbf{W})}\right)G-\delta_{G}G-\eta_{TG}TG,\\
\end{array}
\right.  0\leq t\leq\tilde{\tau},\quad x\in\R,\\
\end{array}
\right.
\end{equation}
where  $\tilde{\tau}$ denotes
the fire period. Following Yatat et al. \cite{Yatat2018}, we have

\begin{itemize}
    \item $\gamma_{G}(\textbf{W})=\displaystyle\frac{\gamma_{G}\times\textbf{W}}{b_{G}+\textbf{W}}$ and
$\gamma_{T}(\textbf{W})=\displaystyle\frac{\gamma_{T}\times\textbf{W}}{b_{T}+\textbf{W}}$
%are annual productions of grass and tree biomasses respectively,
where $\gamma_{G}$  and $\gamma_{T}$ (in yr$^{-1}$) express maximal
growth  of grass and tree biomasses, respectively,  while half saturations
$b_{G}$ and $b_{T}$ (in mm.yr$^{-1}$) determine how quickly they increase with water availability.
    \item $K_{T}(\textbf{W})=\displaystyle\frac{c_T}{1+d_{T}e^{-a_{T}\textbf{W}}}$, where
$c_T$ (in t.ha$^{-1}$) stands for maximum value of the tree biomass
carrying capacity, $a_{T}$ (mm$^{-1}$yr) controls the steepness of
the curve, and $d_{T}$ controls the location of the inflection
point. Similarly,
$K_{G}(\textbf{W})=\displaystyle\frac{c_G}{1+d_{G}e^{-a_{G}\textbf{W}}}$,
where $c_G$ (in t.ha$^{-1}$) denotes the maximum value of the grass
biomass carrying capacity,  $a_{G}$ (mm$^{-1}$yr) controls the
steepness of the curve, and $d_{G}$ controls the location of
the inflection point.
    \item The function $w_G$ is defined by
    \begin{equation}\label{def-omega}
    w_G(G)=\displaystyle\frac{G^{2}}{G^{2}+\alpha_G^{2}},
     \end{equation}
     where $G$, in tons per hectare (t.ha$^{-1}$), is the grass
     biomass and $\alpha_G$ is the value taken by $G$, when the fire intensity is half of its
maximum.
    \item the function $\psi$ is defined by
    \begin{equation}\label{def-psi}
    \psi(T)=\lambda_{fT}^{min} + (\lambda_{fT}^{max}-\lambda_{fT}^{min})e^{-p_TT},
     \end{equation}
where $T$, in tons per hectare (t.ha$^{-1}$), stands for the  tree
biomass, $\lambda_{fT}^{min}$ (in yr$^{-1}$) is the minimal  loss of
tree biomass due to fire in systems with a very large tree biomass,
$\lambda_{fT}^{max}$ (in yr$^{-1}$) is the maximal  loss of
tree/shrub biomass due to fire in open vegetation (e.g. for an
isolated woody individual having its crown within the flame zone),
$p_T$ (in t$^{-1}$.ha) is proportional to the inverse of biomass
suffering an intermediate level of mortality.

\end{itemize}

Assuming that requirement (\ref{tecknical-Assumption}) is satisfied
or, equivalently,
$R_T=\displaystyle\frac{\gamma_{T}(\textbf{W})}{\delta_T}>1$ and
$R_G=\displaystyle\frac{\gamma_{G}(\textbf{W})}{\delta_G}>1$. We set
\begin{equation}\label{scaling}
\left\{
\begin{array}{l}
K_T'=  K_{T}(\textbf{W})(1-1/R_T),\quad K_G'=  K_{G}(\textbf{W})(1-1/R_G),\quad U=T/K_T',\quad V=G/K_G', \\
r=\delta_T(R_T-1)t,\quad \tau=\delta_T(R_T-1)\tilde{\tau},\quad z=x\sqrt{\delta_T(R_T-1)}, \\
\lambda=\displaystyle\frac{\delta_G(R_G-1)}{\delta_T(R_T-1)},\quad
\gamma=\displaystyle\frac{\eta_{TG}K_T'}{\delta_G(R_G-1)}.
\end{array}
\right.
\end{equation}

Hence, with straightforward computations, system (\ref{savanna-RD})
becomes

\begin{equation}\label{savanna-RD-scale}
\left\{
\begin{array}{l}
\left.
\begin{array}{l}
\displaystyle\frac{\partial U}{\partial r}=d_T(\textbf{W})\displaystyle \frac{\partial^2 U}{\partial z^2}+\left(1-U\right)U,\\
\\
\displaystyle \frac{\partial V}{\partial r}=d_G(\textbf{W})\displaystyle \frac{\partial^2 V}{\partial z^2}+\lambda \left(1-V-\gamma U\right)V.\\
\end{array}
\right.  0\leq r\leq\tau, z\in\R,\\
\end{array}
\right.
\end{equation}

Now, letting
\begin{equation}%\label{}
\begin{array}{c}
  t:=r, \quad x:=z,\quad d_T:=d_u,\quad d_G:=d_v,\quad p:=p_TK_T', \\
 \quad  a_{min}:=\lambda_{fT}^{min},\quad
  a_{max}:=\lambda_{fT}^{max},\quad \alpha=\displaystyle\frac{\alpha_G}{K_G'}
\end{array}
\end{equation}
in (\ref{def-omega}), (\ref{def-psi}) and (\ref{savanna-RD-scale}), we recover system
(\ref{cooperation0})-(\ref{mise-a-jour0}).
Furthermore,  scaling (\ref{scaling}) redefines the parameters as follows:
\begin{itemize}
    \item The threshold $\mathcal{R}_0=(1-\eta)\exp(\lambda\tau)$ becomes
    \begin{equation}
    \begin{array}{ccl}
     \tilde{\mathcal{R}}_0 & =& (1-\eta)\exp\left( \displaystyle\frac{\gamma_G(\textbf{W})-\delta_G}{\gamma_T(\textbf{W})-\delta_T}\times (\gamma_T(\textbf{W})-\delta_T)\tilde{\tau}\right),  \\
     & = &  (1-\eta)\exp((\gamma_G(\textbf{W})-\delta_G)\tilde{\tau}).
    \end{array}
\end{equation}
   \item The threshold $\mathcal{R}_1=(1-\eta)\exp(\lambda(1-\gamma)\tau)$ becomes
   \begin{equation}
    \begin{array}{ccl}
     \tilde{\mathcal{R}}_1 & =& (1-\eta)\exp\left((\gamma_G(\textbf{W})-\delta_G)\tilde{\tau}\left(1- \displaystyle\frac{\eta_{TG}K_T'}{\gamma_G(\textbf{W})-\delta_G}\right)\right),  \\
     & = &  (1-\eta)\exp((\gamma_G(\textbf{W})-\delta_G-\eta_{TG}K_T')\tilde{\tau}).
    \end{array}
\end{equation}

\item From $\bar{v}=\displaystyle\frac{\eta}{(1-\eta)(\exp(\lambda\tau)-1)}$ one deduces that
$\bar{G}=\left(1-\displaystyle\frac{\eta}{1-\eta}\times\displaystyle\frac{1}{(\exp((\gamma_G(\textbf{W})-\delta_G)\tilde{\tau})-1)}\right)K_G'$. Let us set $w_{\bar{G}}=\displaystyle\frac{(\bar{G})^2}{(\bar{G})^2+\alpha_G^2}$ and $w_{0}=\displaystyle\frac{(1-\bar{v})^2}{(1-\bar{v})^2+\alpha^2}$. Then, the threshold
$\mathcal{R}_2=(1-a_{max}w_0)\exp(\lambda\tau)$ becomes
\begin{equation}
    \begin{array}{ccl}
     \tilde{\mathcal{R}}_2 &=& (1-\lambda_{fT}^{max}w_{\bar{G}})\exp((\gamma_T(\textbf{W})-\delta_T)\tilde{\tau}).
    \end{array}
\end{equation}
\end{itemize}
Recall that $\mathcal{R}_0$, $\mathcal{R}_1$, $\bar{v}$ and $\mathcal{R}_2$ are related to the normalized system
(\ref{cooperation})-(\ref{mise-a-jour}) while $\tilde{\mathcal{R}}_0$, $\tilde{\mathcal{R}}_1$, $\bar{G}$ and $\tilde{\mathcal{R}}_2$ are related to the original system (\ref{I-savnna})-(\ref{pulsed_swv_eq2}).

In Tables \ref{table-param-1} and \ref{table-param-2}, we summarize the
parameter values that are used for numerical simulations. They are chosen according to \cite{Yatat2017,Yatat2018,PhDYatatDjeumen2018}.

\begin{table}[H]
\centering
\begin{tabular}{cccccc}
    \hline
    $c_{G}$,  t.ha$^{-1}$   & $c_{T}$,  t.ha$^{-1}$  & $b_{G}$, mm.yr$^{-1}$ & $b_{T}$, mm.yr$^{-1}$  &  $a_{G}$, yr$^{-1}$ & $a_{T}$, yr$^{-1}$ \\
    \hline
    $20$ & $450$  & $501$ & $1192$ & $0.0029$ & $0.0045$ \\
    \hline
    $d_{G}$, $-$    & $d_{T}$, $-$ & $\gamma_{G}$, yr$^{-1}$ & $\gamma_{T}$, yr$^{-1}$ & $\delta_{G}$, yr$^{-1}$ & $\delta_{T}$, yr$^{-1}$\\
    \hline
    $14.73$     & $106.7$ & $1.5$ & $2$ & $0.3$ & $0.1$\\
    \hline
    $\eta$, $-$ & $\lambda_{fT}^{min}$, $-$ & $\lambda_{fT}^{max}$, $-$ & $p_T$, t$^{-1}$ha & $\alpha_G$, t.ha$^{-1}$ &$\eta_{TG}$, ha.t$^{-1}$yr$^{-1}$\\
    \hline
    $\textbf{0.7}$ & $0.05$ & $0.6$ & $0.01$ & $\textbf{2}$ & $0.01$  \\
    \hline
\end{tabular}
\caption{Parameter values related to system
(\ref{I-savnna})-(\ref{pulsed_swv_eq2}) at $\textbf{W}=1200$ mm.yr$^{-1}$ and
$\tilde{\tau}=2$ yr. In this setting, the forest equilibrium $\mathbf E_T=(K_T',0)$ is stable
while the grassland equilibrium $\mathbf E_G=(0, \bar{G})$ is
unstable.}\label{table-param-1}
\end{table}

\begin{table}[H]
\centering
\begin{tabular}{cccccc}
    \hline
    $c_{G}$,  t.ha$^{-1}$   & $c_{T}$,  t.ha$^{-1}$  & $b_{G}$, mm.yr$^{-1}$ & $b_{T}$, mm.yr$^{-1}$  &  $a_{G}$, yr$^{-1}$ & $a_{T}$, yr$^{-1}$ \\
    \hline
    $20$ & $450$  & $501$ & $1192$ & $0.0029$ & $0.0045$ \\
    \hline
    $d_{G}$, $-$    & $d_{T}$, $-$ & $\gamma_{G}$, yr$^{-1}$ & $\gamma_{T}$, yr$^{-1}$ & $\delta_{G}$, yr$^{-1}$ & $\delta_{T}$, yr$^{-1}$\\
    \hline
    $14.73$     & $106.7$ & $1.5$ & $2$ & $0.3$ & $0.1$\\
    \hline
    $\eta$, $-$ & $\lambda_{fT}^{min}$, $-$ & $\lambda_{fT}^{max}$, $-$ & $p_T$, t$^{-1}$ha & $\alpha_G$, t.ha$^{-1}$ &$\eta_{TG}$, ha.t$^{-1}$yr$^{-1}$\\
    \hline
    $\textbf{0.1}$ & $0.05$ & $0.6$ & $0.01$ & $\textbf{0.2}$ & $0.01$  \\
    \hline
\end{tabular}
\caption{Parameter values related to system
(\ref{I-savnna})-(\ref{pulsed_swv_eq2}) at $\textbf{W}=450$ mm.yr$^{-1}$ and
$\tilde{\tau}=2$ yr. In this setting, the forest equilibrium $\mathbf E_T=(K_T',0)$ is
unstable while the grassland equilibrium $\mathbf E_G=(0, \bar{G})$ is
stable.}\label{table-param-2}
\end{table}

\section{Another monotone increasing impulsive system}\label{appendix-scaling-grassland}
If, instead of the first coordinates change (\ref{new-variables1}), one
considers
\begin{equation}%\label{new-variables-new1}
\left\{
    \begin{array}{ccl}
      u_n & = & 1-U_n, \\
      v_n & = & V_n,
    \end{array}
\right.
\end{equation}
then the normalized systems
(\ref{cooperation0})-(\ref{mise-a-jour0})  becomes
\begin{equation}\label{cooperation4}
 \left\{%
\begin{array}{lcl}
 \displaystyle\frac{\partial u_n}{\partial t} &=& -u_n(1-u_n)+d_u\displaystyle\frac{\partial^2 u_n}{\partial x^2}, \quad 0\leq t\leq \tau,\quad x\in\mathbb{R},\\
 % & & \\
    \displaystyle\frac{\partial v_n}{\partial t} &=&
   \lambda v_n(1-v_n) - \lambda\gamma v_n(1-u_n)+d_v\displaystyle\frac{\partial^2 v_n}{\partial
  x^2},
\end{array}
\right.
\end{equation}
together with the updating conditions
 \begin{equation}\label{mise-a-jour4}
 \left\{%
\begin{array}{lcl}
u_{n+1}(x,0) &=& w_v(v_{n}(x,\tau))\psi(1-u_n(x,\tau))(1-u_n(x,\tau))+u_n(x,\tau),\\
    v_{n+1}(x,0) &=& (1- \eta)v_{n}(x,\tau).
\end{array}
\right.
\end{equation}
As we mentioned in Subsection \ref{ss}, this is a monotone increasing system and hence, reasoning as before, one
can study the case where the stability of the equilibria is reversed.
\end{document}